\date{\today}
\title{Uniform multicommodity flows in the hypercube with random edge-capacities}
\author{
Colin McDiarmid\\Department of Statistics,\\ University of Oxford,\\ 1 South Parks Road,\\ Oxford, OX1 3TG, UK.\\cmcd@stats.ox.ac.uk\\
\and
Alex Scott\\Mathematical Institute,\\ University of Oxford,\\ Radcliffe Observatory Quarter,\\ Woodstock Road,\\ Oxford, OX2 6GG, UK.\\scott@maths.ox.ac.uk\\ 
\and 
Paul Withers\\Mathematical Institute,\\ University of Oxford,\\ Radcliffe Observatory Quarter,\\ Woodstock Road,\\ Oxford, OX2 6GG, UK.\\withers@maths.ox.ac.uk}

\documentclass[12pt,a4paper,english]{article}
\usepackage{graphicx}
\usepackage{xcolor}
\usepackage{tikz}
\usetikzlibrary{backgrounds}
\usepackage{float}
\floatstyle{boxed} 
\restylefloat{figure}

\newcommand{\Pro}{\ensuremath{\mathbb{P}}}
\newcommand{\Ex}{\ensuremath{\mathbb{E}}}
\newcommand{\Ber}{\ensuremath{\mathrm{Ber}}}
\newcommand{\Bin}{\ensuremath{\mathrm{Bin}}}
\newcommand{\vol}{\ensuremath{\mathrm{vol}}}
\newcommand{\capt}{\ensuremath{\mathrm{cap}}}
\newcommand{\all}{\ensuremath{\mathrm{all}}}
\newcommand{\opp}{\ensuremath{\mathrm{opp}}}
\newcommand{\size}{\ensuremath{\mathrm{size}}}

\newcommand{\minus}{\,\text{--}\,}

\newcommand{\cN}{{\cal N}}
\usepackage{enumerate}
\usepackage[toc]{glossaries}
\makeglossaries
\usepackage{hyperref}
\usepackage{amsmath,amsthm}
\usepackage{amsfonts}
\usepackage{amssymb}
\usepackage{graphicx}
\newtheorem{theorem}{Theorem}
\theoremstyle{plain}

\newtheorem{definition}[theorem]{Definition}

\newtheorem{lemma}[theorem]{Lemma}

\numberwithin{equation}{section}

\begin{document}
\maketitle
\thispagestyle{empty} 
\begin{abstract}
\noindent
We give two results for multicommodity flows in the $d$-dimensional hypercube ${Q}^d$ with independent random edge-capacities distributed like a random variable $C$ where $\Pro[C>0]>1/2$. Firstly, with high probability as $d \rightarrow \infty$, the network can support simultaneous multicommodity flows of volume close to $\Ex[C]$ between all antipodal vertex pairs. Secondly, with high probability, the network can support simultaneous multicommodity flows of volume close to $2^{1-d}\Ex[C]$ between all vertex pairs. Both results are best possible.
\end{abstract}
\newpage


\section{Introduction and statement of results}
 
A \emph{network} $\cN$ consists of an undirected graph together with a capacity $c_e \geq 0$ for each edge $e$. Given a collection $\mathcal{V}$ of unordered pairs of vertices in $\cN$,  
a corresponding \emph{multicommodity flow} $F$  
consists of an $st$-flow $f_{st}$ for each pair $s, t$ in $\mathcal{V}$. 
If each $f_{st}$ has volume $\phi$ we say that $F$ is a \emph{uniform multicommodity flow of volume} $\phi$. We say that $F$ is {\em feasible} if the total flow in each edge~$e$ (with no cancellations) is at most its capacity $c_e$.
The \emph{maximum uniform flow volume} is the maximum value of $\phi$ such that there is a feasible uniform multicommodity flow of volume $\phi$ in $\cN$. 

We investigate multicommodity flows in networks with random edge-capacities. Multicommodity flows are of interest in operational research and combinatorial optimisation and sampling (further background can be found later in this section).  They have been studied extensively from a ``worst-case" perspective, but the ``typical" behaviour of multicommodity flow problems is much less well understood.  In order to address this, we work with a model in which the underlying graph is fixed and the edge-capacities are random.  Aldous, McDiarmid and Scott ~\cite{McD8} studied the case in which the underlying graph is complete, the edge-capacities are independent, each distributed like a given random variable  $C$, and $\mathcal{V}$ is the collection of all unordered pairs of distinct vertices.  They showed that, for a distribution with a finite mean, as $n \to \infty$ the maximum uniform flow volume converges in probability to a constant $\phi_*$, where $\phi_*$ is the unique solution to
\[\Ex[\max\{C-\phi,0\}]= 2 \, \Ex[\max\{\phi-C,0\}].\]
For example $\phi^* = \sqrt{2} -1$ when $C$ is uniformly distributed on $[0,1]$. (See \cite{PWThesis} for related results on complete multipartite graphs.)

In this paper, we consider another very natural example: the $d$-dimensional (hyper)cube $Q^d$, with independent random edge-capacities each distributed like $C$.  In this case there are two natural choices for $\mathcal{V}$.
When $\mathcal{V}$ is the collection of {\bf opp}osite (or antipodal) pairs we denote the maximum uniform flow volume by $\Phi_{\opp}$.
When $\mathcal{V}$ is the collection of {\bf all} pairs of distinct vertices we use $\Phi_{\all}$. 

Finding flows in the cube presents very different problems compared to finding flows in the complete graph. In the complete graph, flow between a pair of vertices that cannot use the direct edge must go at least twice as far, using up at least twice the capacity.  Thus the edge-capacity distribution itself is important, as this determines the proportion of flow that can be routed along efficient paths, with other paths being much less efficient. By contrast, the cube is a sparse graph and the distance between a typical pair of vertices is about $d/2$. This leads to two challenges in finding large-volume multicommodity flows: we must show that there are no local obstructions (vertices or small clusters of vertices that are poorly connected to the rest of the graph, so that little flow can escape the region); and, on the larger scale, we must show that most of the flow can be routed along paths of close to optimal length. It turns out that the capacity required for `local escape' is much less than that required for the efficient flows along long paths: this leads to the limiting optimal flow value depending only on $\Ex[C]$, which establishes a strong form of a conjecture from \cite{McD8}.  Here are our two main theorems, for $\Phi_{\opp}$ and for $\Phi_{\all}$.

\begin{theorem}
\label{th:1}
Let the non-negative random variable $C$ satisfy $\Ex[C]<\infty$ and $\Pro[C>0]>1/2$. 
Then, as $d\rightarrow\infty$,
\[\Phi_{\opp}\rightarrow\Ex [C] \text{  in probability and in expectation}.
\label{eq:th1}
\]
\end{theorem}

\begin{theorem}
\label{th:2}
Let the non-negative random variable $C$ satisfy $\Ex[C]<\infty$ and $\Pro[C>0]>1/2$.
Then, as $d\rightarrow\infty$,
\[2^{d-1}\Phi_{\all} \rightarrow\Ex [C]  \text{  in probability and in expectation}.
\label{eq:th2}
\] 
\end{theorem}

To find flows as required we need very different proof methods from those used in \cite{McD8}. Consider Theorem 1. The idea is, for each antipodal pair $s$ and $t$, to allocate flow only to paths of near the minimum length $d$, and to do so as evenly as possible. To achieve this we allocate a natural fraction of the capacity of each edge to this pair: we find that the middle part of the flow can be handled very efficiently; and we can handle the flow near the ends $s$ and $t$ (to achieve `local escape')  by different less efficient methods, if we allow extra capacity, which turns out to be negligible in the limit.

For a non-negative random variable $C$ with $\Pro[C>0]>\frac 12$ and $\Ex[C]=\infty$, it follows directly from these results by truncation that  $\Phi_{opp}$ and $\Phi_{all}$ both tend to infinity in probability (and so in expectation).
Call an edge \emph{open} when its capacity is $>0$ : 
and call the network \emph{connected} when the subgraph formed by the open edges is connected.
Clearly $\Phi_{\all}=\Phi_{\opp}=0$ if the network contains an isolated vertex.
The condition $\Pro[C>0]>1/2$ is necessary to ensure that, with high probability, the network contains no isolated vertices.
Indeed, when $\Pro(C>0) = 1/2$,  
the probability of the network having no isolated vertex  $\to e^{-1}$ as $d \to \infty$,  and the probability that the network is connected tends to the same limit (see \cite{79}). 

\bf{Plan of the paper. }\mdseries  In the remaining part of this section we give some background on multicommodity flows; and in the next section we give formal definitions of such flows and two probabilistic inequalities which are useful later in the paper.
The rest of the paper is devoted to proving Theorems \ref{th:1} and \ref{th:2}.
The convergence in probability in Theorem~\ref{th:1} can be expressed as two parts:
\begin{equation} \label{ub}
  \mbox{({\em upper bound}) \;\; for } \epsilon>0, \;\; \Pro[\Phi_{\opp}\le(1+\epsilon)\Ex[C]]\rightarrow 1 \mbox{ as } d \to \infty
\end{equation}
and
\begin{equation} \label{lb}
  \mbox{({\em lower bound}) \;\; for } \epsilon>0, \;\; \Pro[\Phi_{\opp}\ge(1-\epsilon)\Ex[C]]\rightarrow 1 \mbox{ as } d \to \infty,
\end{equation}
and similarly for Theorem \ref{th:2}. The upper bounds for both theorems are straightforward and are proved in Section \ref{sub:upper}.
Convergence in expectation is also covered in that section. 
The bulk of the paper is devoted to proving the lower bounds. The lower bound for Theorem \ref{th:1} is proved in Sections \ref{sec:structurelower} - \ref{sec:th1proof},
and for Theorem \ref{th:2} in Section \ref{sec:th2proof}.

\bf{Background on multicommodity flow. }\mdseries\label{3sec:background} Multicommodity flow problems arise in many real-life situations such as flows in transport systems and communication systems, and are studied extensively in Operational Research. See standard texts such as Ahuja, Magnanti and Orlin \cite{AhujaMagnantiOrlin}, Chapter 17, and Winston \cite{114}, Chapters 7 and 8, for further details. The analysis of multicommodity flows in structured networks that model Markov chains is also useful in  establishing bounds on the mixing times for the chains (Sinclair \cite{100}). These can then be used to establish efficient algorithms for the random sampling of combinatorial structures. Such techniques are important in a wide range of problems including, for example, approximating the size of certain sets and combinatorial optimisation by stochastic search. Leighton and Rao \cite{69MuticommodityAlgorithm} used uniform multicommodity flow results to design the first polynomial-time (at most polylog($n$)-times-optimal) approximation algorithms for well-known NP-hard optimization problems such as
graph partitioning, min-cut linear arrangement, crossing number, VLSI layout, and minimum
feedback arc set. 

Most work on analysing multicommodity flows has been directed towards developing algorithms (see for example \cite{Aldous1,42Aldous,Khandawawala,67Leightonetal}). Theoretical studies require an underlying graph with some structure. Alongside the complete graph  \cite{McD8}, the cube is one of the most natural examples, and has applications in the design of randomised routing algorithms for parallel computing (see Valiant \cite{115}) and in random sampling of structures based on binary $d$-tuples.

\section{Definitions and preliminaries}\label{3sec:definitions}
We recall some definitions and notation concerning flows.
Let $G=G(V,E)$ be a graph. We denote the set of neighbours of a vertex $v$ by $\Gamma_G(v)$ and 
for simplicity we use $V \setminus v$ to mean $V \setminus \{v\}$.
Suppose we have a network $\cN$ consisting of an undirected graph $G$ together with a capacity $c_e \geq 0$ for each edge $e$. 
To define a (single commodity) \emph{flow} in $\cN$ we consider each undirected edge $e=uv$ as a pair of directed edges $\overrightarrow{uv}$ and $\overrightarrow{vu}$.  Denote the set of all these directed edges by~$\overrightarrow{E}$;
and give each directed edge the same capacity as the original edge in $\cN$.

Given a function $f:\overrightarrow{E}\rightarrow [0,\infty)$,
the \emph{net outflow} of $f$ at vertex $x$ is defined as 
\[f^+(x)=\sum_{y \in \Gamma_G(x)} ( f(\overrightarrow{xy})- f(\overrightarrow{yx})).\]
And the \emph{net inflow} $f^-(x)$ is $- f^+(x)$. 

For two disjoint, non-empty sets of vertices $S$ and $T$, we say that such a function is a \emph{proper} $ST$ \emph{flow} if it satisfies $f^+(x)=0$ for all $x \notin S \cup T$. A flow that is not proper is \emph{improper}.

A proper $ST$ flow that also satisfies $f(\overrightarrow{xy}) \le c(\overrightarrow{xy})$ for all $\overrightarrow{xy} \,\in \overrightarrow{E}$ is called a a \emph{feasible} $ST$-\emph{flow}. We also assume that for all $e=xy \in E$, 
	either $f(\overrightarrow{xy})=0$ or $f(\overrightarrow{yx})=0$ and we write $f(e)=f(\overrightarrow{xy})+f(\overrightarrow{yx})$.
	
We choose the order of $S$ and $T$ such that $\sum_{x \in S} f^+(x) \ge 0$ and then the \emph{volume} $\vol(f)$ is the magnitude of the flow given by $\vol(f)=\sum_{x \in S} f^+(x)= \sum_{y \in T} f^-(y)$. In the special case when $S=\{s\}$ and $T=\{t\}$ we speak of an $st$-flow. See for example~\cite{AhujaMagnantiOrlin} for further discussion.

Let $\mathcal{V}$ be a non-empty set of unordered pairs of distinct vertices in the network $\cN$.
A \emph{multicommodity flow} $F$ for $\mathcal{V}$ consists of an $st$-flow $f_{st}$ for each  pair $\{s, t\}$ in $\mathcal{V}$ (which we arbitrarily order as $st$, the choice being immaterial as the underlying graph contains no directed edges).
If each $f_{st}$ has volume $\phi$ we say that $F$ is a \emph{uniform multicommodity flow of volume}~$\phi$.  
The \emph{total flow} of $F$ in edge $e$ of $G$ is \mbox{$\sum_{st \in \mathcal{V}}f_{st}(e)$}; 
and $F$ is {\em feasible} if the total flow in each edge $e$ is at most its capacity $c_e$.
The maximum uniform flow volume is the maximum value of $\phi$ such that there is a feasible uniform multicommodity flow of volume $\phi$. 
For a network $\mathcal{N}$ whose underlying graph is $Q^d$ and whose edge-capacities are independent, each distributed like a given random variable $C$ we say $\mathcal{N} \in \mathcal{G}(Q^d,C)$; and in the specific case of $C \sim \Ber(p)$, the Bernoulli distribution with parameter $p$, we say $\mathcal{N} \in \mathcal{G}(Q^d,p)$. 

We need two basic lemmas concerning tail probabilities. 

\begin{lemma}\label{lem:bin}
  Given $1/2<p\le 1$,  there are constants $t>0,\tau>0$ such that 
\[ \Pro[\Bin(d,p) \le t d] \le 2^{-(1+\tau)d}.\]
\end{lemma}
\begin{proof}
Let $X \sim \Bin(d,p)$. For $t>0$ and $x \in (0,1)$, by Markov's inequality
\[
(1-p+xp)^d = \Ex[x^X] \ge x^{td}\Pro[x^X \ge x^{ td}] = x^{ td} \Pro[X \le  td]].
\]
And so
\[
  \Pro[X \le td] \leq \left(x^{- t} (1-p+xp)\right)^d.
\]
We may pick $x \in (0,1)$ such that $1-p+xp< 1/2$, and we then pick $t>0$ such that
$ x^{- t} (1-p+xp) < 1/2 $,
finally we may choose $\tau>0$ so that $x^{- t} (1-p+xp)=2^{-(1+\tau)}$.
\end{proof}
\noindent
The following inequality is a form of Chernoff bound (see for example~\cite{35-AlonSpenser} Theorems A.1.4 and A.1.16).

\begin{lemma}\label{ineq:chern}
Let $X_1,\ldots,X_n$ be independent random variables with all $|X_i|\le 1$,
let $X= X_1+ \cdots +X_n$, and let $Y = X\!-\!\Ex{X}$. Then for each $a \geq 0$,
\[\Pro(Y \geq a) \leq e^{-a^2/2n} \;\;\; \mbox{ and } \;\;\; 
\Pro(Y \leq -a) \leq e^{-a^2/2n}.\]

\end{lemma}


\section{Upper bounds and convergence in expectation}\label{sub:upper}

We need one deterministic lemma. Let $d(u,v)$  denote the number of edges in a shortest path between vertices $u$ and $v$.

\begin{lemma}\label{lem:upper} \label{lem.ubdet}
Let $\cN$ be a network consisting of a graph $G$ together with a capacity $c_e\ge 0$ for each edge $e$; and let $\mathcal{V}$ be any non-empty collection of unordered pairs of distinct vertices of $G$ and let $\phi\ge0$. If there is a feasible uniform  multicommodity flow of volume $\phi$ between all pairs in $\mathcal{V}$ then 
\[ \phi \sum_{\{u,v\} \in \mathcal{V}}d(u,v) \le \sum_{e \in E(G)} c_e. \]
\end{lemma}

\begin{proof}
Let $\{u,v\}$ be a pair in $\mathcal{V}$, suppose it is ordered as $uv$, and consider the flow $f_{uv}$ from $u$ to $v$. This flow can be decomposed as a sum of flows along paths from $u$ to $v$, together perhaps with some flows around cycles (see for example Ahuja, Magnanti and Orlin \cite[page 80]{AhujaMagnantiOrlin}).
Since each of the paths has length at least $d(u,v)$, the total capacity used by $f_{uv}$ is at least $\phi \cdot d(u,v)$.
\end{proof}

  When $G$ is $Q^d$   and $\mathcal{V}$ is the collection of all $2^{d-1}$ antipodal pairs, then by the last lemma
\[ \phi \cdot d 2^{d-1}  \le  \sum_{e \in E(G)} c_e = d2^{d-1} c_{av}, \]
where $c_{av}$ is the average of the edge-capacities of $G$; and so
\begin{equation} 
  \phi \le c_{av}.\label{eq:exp1}
\end{equation}
Observe that for each vertex $u$, $\; 2^{-d} \sum_{v} d(u,v) = d/2$.  Thus
if we take $\mathcal{V}$ as the collection of all pairs of distinct vertices then 
\[ \sum_{\{u,v\} \in \mathcal{V}} d(u,v) = \frac{1}{2} \sum_{u, v} d(u,v) = \frac12 \, 2^{2d} \, \frac{d}{2},\]
and so as above
\[ \phi \cdot d 2^{2d-2} \le d2^{d-1} c_{av},\]
that is
\begin{equation}
  \phi \cdot 2^{d-1} \le c_{av}. \label{eq:exp2}
\end{equation}

When we have random edge-capacities,  $C_{av}$ is the mean of $d2^{d-1}$ iid (independent, identically distributed) random variables with finite mean $\Ex[C]$,
and so by the weak law of large numbers,  
given $\epsilon>0$, $\Pro[C_{av}>(1+\epsilon)\Ex[C]]\rightarrow0$ as $d \to \infty$;
and the upper bounds in Theorems \ref{th:1} and \ref{th:2} as in~(\ref{ub}) follow from~(\ref{eq:exp1}) and (\ref{eq:exp2}).
\smallskip

Now consider expectations. Observe that always $\Phi_{\opp} \ge 0$, and from~(\ref{eq:exp1}) we have $\Ex[\Phi_{\opp}]\le \Ex[C]$.  Thus once we have proved that
$\Phi_{\opp}\rightarrow \Ex [C]$ in probability it follows that $\Ex[\Phi_{\opp}] \rightarrow \Ex [C]$.
Similarly, from~(\ref{eq:exp2}) we see that $2^{d-1} \Ex[\Phi_{\all}]\le \Ex[C]$.  Thus once we have proved that
$2^{d-1} \Phi_{\all}\rightarrow \Ex [C]$ in probability it follows that $2^{d-1} \Ex[\Phi_{\all}] \rightarrow \Ex [C]$.
Thus it remains to prove the lower bounds, as in~(\ref{lb}), for convergence in probability.


\section{Antipodal flows: overview of lower bound proof}\label{sec:structurelower}

We break the proof of the lower bound (\ref{lb})~in Theorem \ref{th:1} down into $2^{d-1}$ separate parts, each of which concerns the flow of a single commodity between a pair of antipodal vertices. For each such pair we allocate a portion of the capacity of each edge according to scaling factors described in the next paragraphs.
We find that, for each pair, there is only a very small probability that there fails to be a
feasible flow of volume about $\Ex[C]$ in the restricted network (see Lemma \ref{lem:0}). Then by taking the union bound, with high probability such flows exist simultaneously for all antipodal vertex pairs. When we superimpose these flows we need to sum the capacity used by all the separate flows; and we show that for every edge $e$ the total is at most $(1+o(1))c_e$. Thus the theorem follows by rescaling flows and capacities.
\smallskip

In order to introduce the capacity scaling, let us first define the vertex and edge `layers' from a given source vertex $u$ in $Q^d$.
For $m=0,1,\ldots,d$ the \emph{vertex layer} $V_m(u)$ consists of all the vertices at distance $m$
from $u$; and for $m=1,\ldots,d$ the \emph{edge layer} $E_m(u)$ consists of all the edges between vertices in $V_{m-1}(u)$ and $V_m(u)$.
We shall often take $\mathbf{0}$ (the $d$-vector of 0's) as a representative vertex
(note that the cube is vertex-transitive).  Write $V_m$ for $V_m(\mathbf{0})$ and $E_m$ for $E_m(\mathbf{0})$,
and note that $|V_m|=\binom{d}{m}$ and $|E_m|=m\binom{d}{m}$.

The capacity of each edge is divided amongst the $2^{d-1}$ sub-problems as follows.
We fix a constant $1/2<\kappa<1$ and we define
\[\ell= \ell(d)=\left\lfloor d^{\kappa}\right\rfloor.\] 
Here we think of the $\ell$ as {\bf l}ocal.  The quantity $\kappa$   stays fixed throughout the paper. For a specific pair of antipodal vertices $u,\overline{u}$ the scaling is achieved in two stages. Firstly, for each $m$,  each edge $e$ in layer $E_m(u)$  is given a capacity of $c_e/|E_m|$. We denote the network constructed so far by $\mathcal{N}(u)$. For the second stage we choose a (large) constant $M$
and scale up the capacities on the first and last $\ell+2$ edge layers,
so the capacity of an edge in layer $E_m$ is
\begin{equation}
\label{eq:capacity}
\mbox{Scaled capacity} = \begin{cases} c_e/|E_m| &\mbox{if } \ell+3\le m \le d-\ell-2 \\
M \cdot c_e/|E_m| & \mbox{if } m \le \ell+2 \mbox{ or } m\ge d-\ell-1. \end{cases}
\end{equation}

The network with these scaled capacities is denoted by $\cN^M(u)$, and is illustrated for the case $u=\bf{0}$ in Figure \ref{fig:1}.

\floatstyle{plain}
\restylefloat{figure}	
	
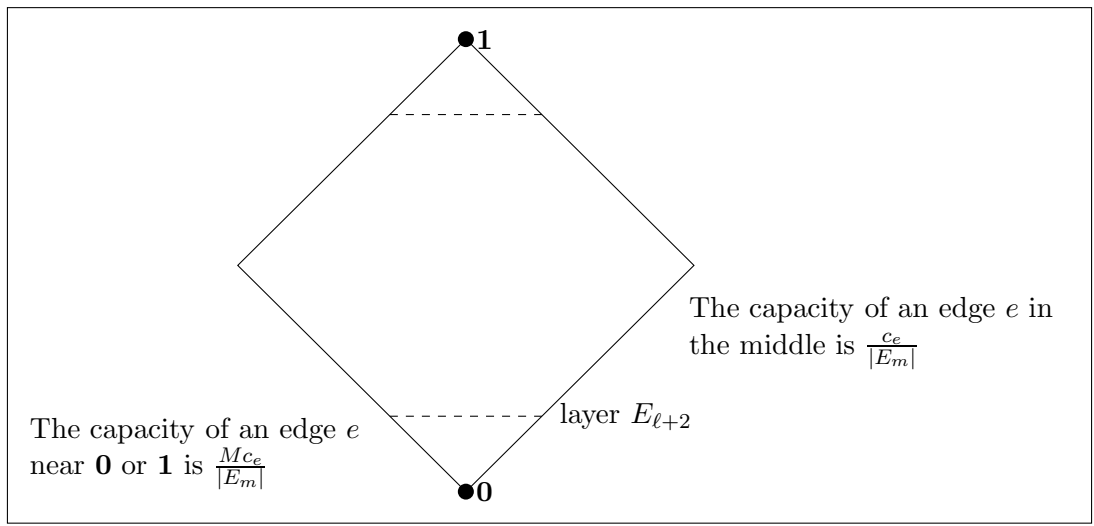
\begin{figure}[h]
	\centering
		\small
\begin{tikzpicture}[show background rectangle]
\fill[black] (0,0) circle(3pt) node[right]{$\bf{0}$};
\fill[black] (0,6) circle(3pt) node[right]{$\bf{1}$};
\draw (0,0) -- (-3,3) -- (0,6) -- (3,3) -- (0,0);
\draw[dashed] (-1,1) -- (1,1);
\draw[dashed] (-1,5) -- (1,5);
\draw (1.1,1) node[right, text width=5cm]{layer $E_{\ell+2}$};
\draw (-0.6,0.5) node[left, text width=5cm]{The capacity of an edge $e$ near $\bf{0}$ or $\bf{1}$ is $\frac{Mc_e}{|E_m|}$};
\draw (2.8,2.1) node[right, text width=5cm]{The capacity of an edge $e$ in the middle is $\frac{c_e}{|E_m|}$};
\end{tikzpicture}
\normalsize
\caption{\mdseries  The scaled network $\cN^M(\mathbf{0})$ illustrated as layers. \normalsize}
	\label{fig:1}
\end{figure}	
\floatstyle{boxed}
\restylefloat{figure}
The first scaling  provides sufficient capacity for the `middle part' of the flow, where we are not near $u$ or $\overline{u}$.
We shall see that the total of the capacities scaled by the first factor for a given edge $e$, summed over all the $2^{d-1}$ antipodal pairs, is exactly $c_e$. 

The second scaling factor $M$ is introduced to enable flows to `escape' from vertex $u$ and reach the boundary layer $V_\ell(u)$ (and similarly for $\overline{u}$). Due to local constraints, these flows may require capacity in an edge $e$ 
in the first $\ell+2$ edge-layers which is considerably greater than $c_e/|E_m|$, so the factor $M$ must be chosen sufficiently large. However, the impact of this `profligate' use of capacity close to a source or sink turns out to be negligible, since $\kappa<1$ and thus $\ell = o(d)$. We shall see that we need the other bound on $\kappa$, namely that $\kappa >1/2$, in the proof of Lemma \ref{lem:n3} (further details are given at the end of the proof of that lemma).

The following lemma is the main step in the proof of Theorem \ref{th:1}.
\begin{lemma}\label{lem:0}
  Given $\epsilon > 0$ there are constant $M$ and $\rho>0$
  such that the following holds.
Let $\phi(u)$ be the maximum feasible flow volume between vertices $u$ and $\overline{u}$ in $\cN^M(u)$. Then as $d \to \infty$,
\begin{equation}
\Pro[\phi(u) < (1-\epsilon) \Ex[C] \text{ for some } u \text{ in } Q^d] =O( 2^{-\rho d}). 
\label{eq:lem01}
\end{equation}
\end{lemma}	
In Sections \ref{sec:antipodal1} and \ref{sec:antipodal2} we establish preliminary results concerning flow close to the source and concerning the middle part of the flow, and in Section \ref{sec:th1proof} we complete the proof of Lemma~\ref{lem:0}.
It is helpful to work with `balanced' and `nearly balanced' flows. Given disjoint non-empty sets of vertices $S$ and $T$ in a network $\cN$,
and an $ST$- flow $f$,
 we say that $f$ is \emph{balanced} if the net outflow at each vertex in $S$ is $\vol(f)/|S|$ and the net inflow at each vertex in $T$ is $\vol(f)/|T|$. Given $\mu >0$, we say that $f$ is $\mu-$\emph{near-balanced} if 
	\[\sum_{v \in S}\left| f^+(v) - \frac{\vol(f)}{|S|}\right| +
	\sum_{v \in T}\left|f^-(v)-\frac{\vol(f)}{|T|}\right| \leq \mu \,\vol(f).\]


\section{Antipodal flows close to a source}\label{sec:antipodal1}

Our aim in this section is to show that with high probability
(that is, with probability $\to 1$ as $d \to \infty$),
for each vertex $u$ there is a balanced flow 
of volume $\Ex[C]$ from $u$ to $V_{\ell}(u)$ in $\cN^M(u)$, when capacities are scaled by a suitable factor~$M$.

We begin with the special case when capacities take only values 0 or~1. The general case will follow easily.
Given $0 \leq p \leq 1$ we let $Q_p=(Q^d)_p$ be the random subgraph of $Q^d$ where the edges appear independently with probability $p$.  We also think of this as a random network based on $Q^d$ where the edge-capacity $C$ satisfies $\Pro(C=1)=p$ and $\Pro(C=0)=1-p$.

For this case we show that (a) with high probability all vertices are suitably `locally connected', (b) when this holds there must be a balanced flow of volume $\Ex[C]$ from $u$ to $V_1(u)$ when capacities are scaled by a suitable constant $M_1$, and (c) by scaling capacities by $M=7 M_1$ we may find a flow from $u$ to $V_{\ell}(u)$ as required. Parts (b) and (c) are deterministic.
\medskip

The cube $Q^d$ can equivalently be defined by representing each vertex by a distinct subset of $[d]=\{1,2,\dots,d\}$ with two vertices being adjacent when their symmetric difference is a singleton. 

Each vertex in the cube $Q=Q^d$ has degree $d$. Let $u$ be a vertex of $Q$ and let $v \in \Gamma_Q(u)$. Then the edges in $Q$ between $\Gamma_Q(u) \setminus v$ and $\Gamma_Q(v) \setminus u$ form a perfect matching, of size $d-1$.  Each edge is of the form $w_1w_2$, where $w_1 \in V_1(u)$ and $w_2 \in V_2(u)$.  These edges $w_1w_2$, together with the edges $uw_1$ and $w_2v$ form $d-1$ internally vertex-disjoint $uv$-paths $uw_1w_2v$.
See figure~\ref{fig:d22n}, which illustrates the case  $u=\emptyset$ (or $\mathbf 0$) and $v=\{1\}$.

\floatstyle{plain}
\restylefloat{figure}
\begin{figure}[h]
\small
\begin{tikzpicture}[show background rectangle]
\fill[black] (0,0) circle(3pt) node[below]{${ u=\emptyset}$};
\fill[black] (-2,2) circle(3pt)node[left]{${\scriptstyle v=\{1\}}$};
\fill[black] (-1,2) circle(3pt)node[left]{${\scriptstyle\{2\}}$};
\fill[black] (-0.5,2) circle(3pt)node[right]{${\scriptstyle\{3\}}\dots$};
\fill[black] (1,2) circle(3pt)node[right]{${\scriptstyle\{d\}}$};
\fill[black] (-1.5,4) circle(3pt)node[above left]{${\scriptstyle\{1,2\}}$};
\fill[black] (-1,4) circle(3pt)node[above right]{${\scriptstyle\{1,3\}}\dots$};
\fill[black] (0.5,4) circle(3pt)node[above right]{${\scriptstyle\{1,d\}}$};
\draw (0,2) ellipse (3.7cm and 0.7cm);
\draw (0,4) ellipse (3.7cm and 0.7cm);
\normalsize
\draw (1.8,2)node[right]{\small ${\Gamma_Q(u)} $};
\draw (1.4,4)node[right]{\small ${\Gamma_Q(v)\!\setminus\!u}$};
\small
\draw[gray, dashed] (0,0) -- (-1,2);
\draw[gray, dashed] (0,0) -- (-0.5,2);
\draw[gray, dashed] (0,0) -- (1,2);
\draw[gray, dashed] (-2,2) -- (-1.5,4);
\draw[gray, dashed] (-2,2) -- (-1,4);
\draw[gray, dashed] (-2,2) -- (0.5,4);
\draw[thick] (-1.5,4) -- (-1,2);
\draw[thick] (-1,4) -- (-0.5,2);
\draw[thick] (0.5,4) -- (1,2);
\draw (4,2) node[right,text width=5.5cm]
{ Each vertex in $\Gamma_Q(v) \setminus u$ is of the form $\{1,j\}$ for some $j \in \{2,\dots,d\}$ and has a unique neighbour $\{j\}$ in $\Gamma_Q(u) \setminus v$.};
\end{tikzpicture}
\normalsize
\caption{\mdseries Matching between $\Gamma_Q(v) \setminus u$ and $\Gamma_Q(u) \setminus v$ with corresponding $uv$ paths. \normalsize}
\label{fig:d22n}
\end{figure}
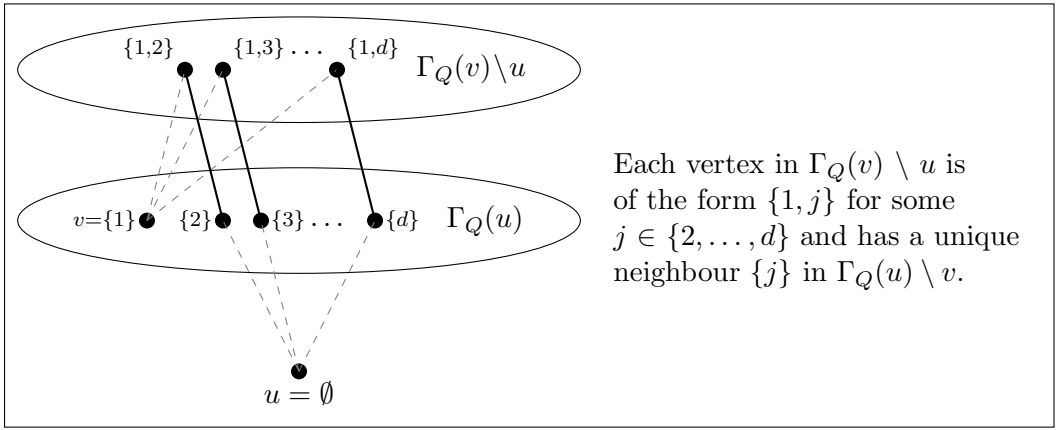
\floatstyle{boxed}
\restylefloat{figure}
We say a subgraph $R$ of $G$ is a \emph{spanning} subgraph if $V(R)=V(G)$.
\begin{definition}\label{def:alp}
Let $0<\alpha<1$ be a constant. Let $R$ be a spanning subgraph of $Q=Q^d$. We say a vertex $u\in R$ is $\alpha$-\emph{locally-connected} in $R$ if

\begin{enumerate}
\item
   its degree in $R$ is at least $\alpha d$, that is $|\Gamma_R(u)|\geq\alpha d$,
\item
  for each $v \in \Gamma_Q(u)$ there exists a matching in $R$ of size at least $\alpha d$ between
$\Gamma_Q(u) \setminus v$ and $\Gamma_Q(v) \setminus u$, 
\item
  For all but at most $\alpha^{-1}$ vertices $v \in \Gamma_Q(u)$ there exists a matching in $R$ of size at least $\alpha d$ between $\Gamma_R(u)\setminus v$ and $\Gamma_R(v)\setminus u$ (and so there are at least $ \alpha d$ internally vertex-disjoint $uv$ paths of length 3 in the graph $R$).  
\end{enumerate}
If all vertices in $R$ are $\alpha$-locally-connected in $R$ then we say that $R$ is \emph{$\alpha$-locally-connected}. A vertex that is not $\alpha$-locally-connected  is $\alpha$-\emph{poorly-connected}. We denote the set of vertices that fail criterion $i$ above as $T^{\alpha}_i$ and the set of $\alpha$-poorly-connected vertices as $T^{\alpha}=T^{\alpha}_1\cup T^{\alpha}_2 \cup T^{\alpha}_3$.
\end{definition}

Once $d>\alpha^{-1}$ the first condition of the definition is implied by the third (so $T^{\alpha}_1 \subseteq T^{\alpha}_3$). We now show that $Q_p$ is $\alpha$-locally-connected with high probability for sufficiently small $\alpha>0$. 

\begin{lemma}\label{sublem:1}
Given $1/2<p<1$, there  are constants $\alpha>0, \rho>0$ such that 
$Q_p$ is $\alpha$-locally-connected with probability at least $1-2^{-\rho d}$.
\end{lemma}
\begin{proof}
Let $t>0,\tau>0$ be as in Lemma \ref{lem:bin} and set $\alpha=\min{t/2,\left\lceil 34/p^4t\right\rceil^{-1}}$. Consider a fixed vertex $u$. We check whether $u$ is $\alpha$-locally-connected.  

\bfseries\noindent 
First condition.
\mdseries
Since $\deg_{Q_p}(u)\sim \Bin(d,p)$,
\begin{equation}
\Pro[\deg_{Q_p}(u)\le td]\le 2^{-(1+\tau)d}.
\label{eqn.1}
\end{equation}
\bfseries\noindent
Second condition.
\mdseries
  For each $v \in \Gamma_Q(u)$ the edges in $Q$ between
  $\Gamma_Q(u) \setminus v$ and $\Gamma_Q(v) \setminus u$ form a matching of size $d-1$.
  The number, $X_{v}$, of these edges that are present in $Q_p$ has distribution  $\Bin(d-1,p)$. Hence 
  \mbox{$\Pro[X_{v} \le t (d-1)] \le 2^{-(1+\tau)(d-1)}$.} Thus for $d$ sufficiently large,
\begin{equation} \label{eqn.2}
  \Pro\left[ X_v< td/2 \mbox{ for some } v \in \Gamma_Q(u) \right]
  \leq d \, 2^{-(1+\tau)(d-1)} \le 2^{-(1+\tau/2)d}.
\end{equation}

\bfseries\noindent
Third condition.
\mdseries
We condition on $d_{Q_p}(u)\ge td$. Let $S \subseteq  \Gamma_Q(u)$ with $|S| \geq t d$, and condition on $\Gamma_{Q_p}(u)= S$.  Let $k \in \mathbb{N}$ and $S^* \subset V_1(u)$ with $|S^*|=k$. For each $v \in  S^*$ let $Y_v$ be the number of paths in $Q_p$ of the form
  $uw_1w_2v$ where $w_1 \in S \setminus S^*$ and $w_2 \in V_2(u)$. Consider as an example (using the subset of $[d]$ notation introduced in Section \ref{sec:antipodal1} to denote our vertices) $u=\emptyset$,  $v=\{1\}$ and $x\in S$ with $x=\{i\}$: the only possible path is $\emptyset-\{i\}-\{1,i\}-\{1\}$, and the last two edges in this path are disjoint from paths from any other choice of $v$ or $x$. In general each edge in $E_2(u)$ is in at most one 3-edge path $uw_1w_2v$ where $w_1 \in S \setminus S^*$ and $v \in S^*$. Also $|S \setminus S^*|>\frac{td}2$ for large $d$ so
  each $Y_v$ is stochastically at least  $\Bin(\left\lceil \frac{td}2\right\rceil,p^2)$, and the random variables $Y_v$ for
  $v \in S^*$ are independent. (For two random variables $X$ and $Y$, we say $X$ is \emph{stochastically} less than $Y$ if $\Pro[X \le x] \ge \Pro[Y \le x]$ for all possible $x$.) 
  
  We use the Chernoff Bounds as in Lemma \ref{ineq:chern}, and set $s=p^2 \left\lceil \frac{td}2\right\rceil=\Ex[\Bin\left(\left\lceil \frac{td}2\right\rceil,p^2\right)]$ to get:

  \[ \Pro \left[Y_v \le \frac14 p^2 td \right] \leq \Pro\left[\Bin\left(\left\lceil \frac{td}2\right\rceil,p^2\right)\le \frac s2\right] \leq e^{-\frac1{16} p^4 t d} =: q. \]
  
Let $A(S^*)$ be the event that $Y_v\le \frac14 p^2 td$ for each vertex in $S^*$. Thus $\Pro[A(S^*)]\le q^k$.  There are $\binom{d}{k}$ possible choices of the set $S^*$ so, by the union bound, for sufficiently large $d$, the probability that $A(S^*)$ is true for one of these $k$-sets  is at most
\[\binom{d}{k}q^k \le d^ke^{-\frac1{16}p^4tkd}\le e^{-\frac1{17}p^4tkd}\]
So we choose $k =\left\lceil 34/p^4t\right\rceil$ and with failure probability at most \mbox{$e^{-2d}$}  there is no subset of $V_1(u)$ of size $k$ with every vertex having the property $Y_v \le \frac14 p^2 td.$ Now we can remove the conditioning that $\Gamma_{Q_p}(u)= S$, and we see that, conditional on $d_{Q_p}(u)\ge td$, the probability that $A(S^*)$ holds for some $k$-subset of $V_1(u)$ is at most $e^{-2d}$.

Now $\Pro[u \in T^{\alpha}]\le \Pro[u \in T_1^{\alpha}]+\Pro[u \in T_2^{\alpha}]+\Pro[u \in T_3^{\alpha}|u \notin T_1^{\alpha}]$. 
Using equation (\ref{eqn.1}) $\Pro[u \in T^{\alpha}_1]\le 2^{-(1+\tau)d}$, and for large $d$, from equation (\ref{eqn.2}) $\Pro[u \in T^{\alpha}_2]\le 2^{-(1+\tau/2)d}$ and from the previous section $\Pro[u \in T^{\alpha}_3|u \notin T_1^{\alpha}]\le e^{-2d}$. Thus, for large $d$, the probability that $u$ is not $\alpha$-locally-connected is at most $2^{-(1+\rho)d}$ where $\rho=\tau/3$; and by the union bound, $Q_p$ is $\alpha$-locally-connected with probability at least $1-2^{-\rho d}$.
\end{proof}

We denote by $B^{M}_{r}(u)$ 
the ball of radius $r$ centred on vertex $u$ with the capacities scaled as in $\cN^M(u)$. Thus $B^{M}_r(u)$ is $\cN^M(u)$ restricted to the vertex layers up to $V_r(u)$. We say $\cN$ has the \emph{local escape property} with parameter $M$ if, for each vertex $u$, we can route a balanced feasible flow of volume $1$ between $u$ and $\Gamma_Q(u)$ using only paths in  $B_{3}^M(u)$. 

We now prove two deterministic lemmas.

\begin{lemma}\label{lem:local}
Given $0<\alpha<1$ there exists a constant $M_1=M_1(\alpha)>0$ such that the following holds. 
Let the spanning subgraph $R$ of $Q$ be $\alpha$-locally-connected, and let $\mathcal{N}$ be the network formed by giving capacity 1 to each edge of $R$. Then  $\mathcal{N}$ has the local escape property with parameter $M_1$.
\end{lemma}
\begin{proof}

Consider vertex $u$. Let $S_1=S_1(u)=\Gamma_R(u)$,  let $S_3=S_3(u)$ be the set of vertices in $\Gamma_Q(u) \setminus S_1$ for which the number of paths of length 3 in $R$ to $u $ is greater than $\alpha d$,  and denote the remainder ($\Gamma_Q(u)\setminus (S_1 \cup S_3)$) by $S^*=S^*(u)$. Since $R$ is $\alpha$-locally-connected we have $|S_1| \ge \alpha d$ and $|S^*|\le\left\lfloor \alpha^{-1}\right\rfloor =: N$. We now route flows as follows.

For $v \in S_1$ we route a flow of $d^{-1}$ along edge $u$$v$. For $v \in S_3$ we route a flow of $d^{-1}$ split evenly between $\left\lceil \alpha d\right\rceil $ of the paths of length 3 from $u$ to $v$ (chosen arbitrarily). The total of these flows is at most $\left( \alpha d\right)^{-1}d^{-1}<\alpha^{-1}|E_2|^{-1}$ in any edge in $E_2=E_2(u)$ and at most $(1+\alpha^{-1})d^{-1}$ in any edge in layer $E_1=E_1(u)$. Thus we can route flows of volume $d^{-1}$ to all vertices in $S_1 \cup S_3$ in $B_{u}^M(2)$ for $M_0=1+ \alpha^{-1}$. 

For $v \in S^*$ we know there is a  matching in $R$ of size at least $\alpha d$ between $\Gamma_Q(u) \setminus v$ and $\Gamma_Q(v) \setminus u$, so there is a matching in $R$ of size at least $\alpha d-|S^*(u)|-|S^*(v)|\ge \alpha d - 2N$ between $S_1(u)\cup S_3(u)$ and $S_1(v)\cup S_3(v)$. For large enough $d$, we may pick a matching of size  $\left\lceil \alpha d/2\right\rceil$, say $v_1w_1,v_2w_2,\dots$ where $v_i \in S_1(v)\cup S_3(v)$ and $w_i \in S_1(u)\cup S_3(u)$.  We  route a flow of volume $\left\lceil  \alpha d/2\right\rceil^{-1}d^{-1}$ along each of these edges. We route the flows from $u$ to the $v_i$ by scaling the flows already found above by a factor $1+\left\lceil \alpha d/2\right\rceil^{-1}$ and from the $w_i$ to $v$ using the routes already found by the same method and using the same scaling factor. There are at most $N$ vertices in $S^*$ so the volume from all the flows from $u$ to $S^*$ is at most $NM_0\left\lceil  \alpha d/2\right\rceil^{-1}d^{-2}\le \frac{M_0N}{\alpha}|E_3|^{-1}$ in any edge in $E_3=E_3(v)$ and $NM_0\left\lceil  \alpha d/2\right\rceil^{-1}d^{-1}\le \frac{M_0N}{\alpha}|E_2|^{-1}$ in any edge in layer $E_2$ and $N\left\lceil  \alpha d\right\rceil^{-1}d^{-1}\le \frac{M_0N}{\alpha}|E_1|^{-1}$ in any edge in $E_1$. Thus we can route flows of volume 1 to all vertices in $\Gamma_Q(u)$ in $B_{3}^{M_1}(u)$ for $M_1= \frac{M_0N}{\alpha}$.
\end{proof}

\begin{lemma}\label{lem:prop}
 
Let $R$ be a spanning subgraph of $Q$ and let $\mathcal{N}$ be the network formed by giving a capacity of 1 to each edge of $R$. Assume that $\mathcal{N}$ has the local escape property with parameter $M_1$. Then, for $d$ sufficiently large, for each $u \in V(Q)$ there exists a balanced flow of volume 1 between $u$ and $V_{\ell}(u)$ in $B_{u}^{M}({\ell}+2)$, where $M=7 M_1$.
\end{lemma}
\begin{proof}
Fix $u \in V(Q)$. Let $1 \le m \le \ell$ and let $v \in V_{m-1}(u)$. By our assumption on 
$\mathcal{N}$, there is a balanced flow of volume $\frac{d}{d-m+1}|V_{m-1}|^{-1}$ from  $v$ to all of its neighbours in $Q$  in the ball $B^{M_2}_{3}(v)$ where $M_2=\frac{d}{d-m+1}|V_{m-1}|^{-1}M_1$. We decompose this flow into flows along paths of length at most 7 and only consider those flows to neighbours of $v$ in $V_m$. This results in separate flows of volume $\frac{1}{d-m+1}|V_{m-1}|^{-1}$ from $v$ to each of its neighbours in $V_m$. Repeating this process for every vertex in $V_{m-1}$ gives us a flow to every vertex in $V_m$ of volume $\frac{m}{d-m+1}|V_{m-1}|^{-1}=|V_m|^{-1}$. Thus a balanced flow of volume 1 exists between $V_{m-1}$ and $V_m$ if each edge has the capacity given by the addition of the capacities of that edge in the balls $B^{M_1}_{3}(v)$ for all $v \in V_{m-1}$ scaled by the factor $\frac{d}{d-m+1}|V_{m-1}|^{-1}$. We repeat this process for all vertex layers 1 to $\ell$ to get a balanced flow of volume 1 from $u$ to $V_{\ell}(u)$.

To calculate an upper bound to the total capacity required in edge $e=xy$ ($x \in V_{m-1}, y \in V_m$) we need to consider all of the balls centred on vertices a distance at most 2 from $x$ or $y$. The capacity required in $e$ for one of these balls is the product of the scaling factor applied to the ball (the factor $\frac{d}{d-m+1}|V_{m-1}|^{-1}M_1$ for a ball centred on a vertex in $V_{m-1}$) and the factor applied to edge $e$ in the ball. ($\binom{d}{k}^{-1}$ if $e$ is in edge-layer $k$ in the ball.) The vertices within a distance 2 of $x$ can be partitioned into 6 sets.
\begin{align*}
S_1 &= \{x\},~|S_1|=1\\
S_2 &= \{z \in V_{m-2}:d_Q(xz)=1\},~\text{where }|S_2|=m-1\\
S_3 &= \{z \in V_{m}:d_Q(xz)=1\},~\text{where }|S_3|=d-m+1\\
S_4 &= \{z \in V_{m-3}:d_Q(xz)=2\},~\text{where }|S_4|=\binom{m-1}2\\
S_5 &= \{z \in V_{m-1}:d_Q(xz)=2\},~\text{where }|S_5|=(m-1)(d-m+1)\\
S_2 &= \{z \in V_{m+1}:d_Q(xz)=2\},~\text{where }|S_6|=\binom{d-m+1}2
\end{align*}
We denote by $C_i$ the total capacity required in edge $e$ for the balls centred on vertices in set $S_i$ and
\begin{align*}
C_1 &= |S_1| \frac{d}{d-m+1}|V_{m-1}|^{-1}\frac{M_1}{d}=M_1 (|V_{m-1}|d)^{-1}(1+O(m/d))\\
C_2 &= |S_2| \frac{d}{d-m+2}|V_{m-2}|^{-1}M_1\binom{d}{2}^{-1}=2M_1 (|V_{m-1}|d)^{-1}(1+O(m/d))\\
C_3 &= |S_3| \frac{d}{d-m}|V_{m}|^{-1}M_1\binom{d}{2}^{-1}=O\left( \frac md M_1 (|V_{m-1}|d)^{-1}\right)\\
C_4 &= |S_4| \frac{d}{d-m+3}|V_{m-3}|^{-1}M_1\binom{d}{3}^{-1}=3M_1 (|V_{m-1}|d)^{-1}(1+O(m/d))\\
C_5 &= |S_5| \frac{d}{d-m+1}|V_{m-1}|^{-1}M_1\binom{d}{3}^{-1}=O\left( \frac md M_1 (|V_{m-1}|d)^{-1}\right)\\
C_6 &= |S_6| \frac{d}{d-m-1}|V_{m+1}|^{-1}M_1\binom{d}{3}^{-1}=O\left( \frac md M_1 (|V_{m-1}|d)^{-1}\right)
\end{align*}
And the sum of these is $\sum C_i=6M_1 (|V_{m-1}|d)^{-1}(1+O(m/d))$. We do a similar exercise for balls centred on vertices a distance at most 2 from $y$ and get the total capacity required for all these balls is $6M_1 (|V_{m}|d)^{-1}(1+O(m/d))$. We note that $|V_m|^{-1}=O\left(m/d |V_{m-1}|^{-1}\right)$, so the total capacity required in edge $e$ from all balls centred on vertices a distance at most 2 from either $x$ or $y$ is $6M_1 (|V_{m-1}|d)^{-1}(1+O(m/d))$. Putting $M=7M_1$ gives the result.
\end{proof}
\noindent
Putting Lemmas \ref{sublem:1}, \ref{lem:local} and \ref{lem:prop} together, we obtain the following lemma.
\vspace{8pt}

\begin{lemma}
\label{lem:main1p}
Given $1/2<p<1$ ,  there exist  constants $M>0$, and $\rho >0$  such that for $d$ sufficiently large, with failure probability at most $2^{-\rho d}$, for all vertices $u$ in $Q_p$ there exists a balanced flow of volume 1 in the ball $B^M_{{\ell}+2}(u)$ from $u$ to $V_{\ell}(u)$.\qed
\end{lemma}
We can now prove the main result of this section, Lemma \ref{lem:main1}. In Lemma \ref{lem:main1p} we proved the result in the special case when $C \sim \Ber(p)$, the Bernoulli distribution with parameter $p$.  

\begin{lemma}\label{lem:main1}
For a random variable $C$ with $\Pro[C>0]>1/2$ there exist constants $ M>0$ and $\rho>0$ such that the following holds with failure probability at most $2^{-\rho d }$.  For $\mathcal{N} \in \mathcal{G}(Q,C)$, for all vertices $u$ there exists a balanced flow of volume $\Ex[C]$ in the ball $B^M_{{\ell}+2}(u)$ from $u$ to $V_{\ell}(u)$.
\end{lemma}
\begin{proof}
Since $\Pro[C>0] >1/2$, there is $c^*>0 $ such that $\Pro[C \ge c^*]>1/2$. We denote $\Pro[C \ge c^*]$ by $p^*$. We now consider the network $\mathcal{N}^*$ which is the network $\mathcal{N}$ with edge-capacities reduced as follows. For an edge $e$ with $C_{\mathcal{N}}(e)<c^*$ we put $C_{\mathcal{N}^*}(e)=0$ and for an edge $e$ with $C_{\mathcal{N}}(e)\ge c^*$ we put $C_{\mathcal{N}^*}(e)=c^*$. Thus the edge-capacities of $\mathcal{N}^*$ have distribution $c^*\Ber(p^*)$ and the capacity of each edge in $G^*$ is at most its capacity in $G$. From Lemma \ref{lem:main1p} we know that there exists $M^{\prime}$ and $\rho>0$ such that with failure probability at most $2^{-\rho d}$, for all vertices $u$, there exists a balanced flow of volume 1 in the ball $B^{M^{\prime}}_{{\ell}+2}(u)$ from $u$ to $V_{\ell}(u)$. Putting $M=\frac{{M^{\prime}}\Ex[C]}{c^*}$ we get the result.
\end{proof}



%
\section{Antipodal flows in the middle part of a flow}\label{sec:antipodal2}

Our aim in this section is to show that, with failure probability $O(e^{-d^2})$, there is a $d^{-2}$-near-balanced flow of volume $(1+\epsilon)^{-1}\Ex[C]$ from $V_{\ell}(\mathbf{0})$ to $V_{\ell}({\mathbf{1}})$ in our scaled network $\mathcal{N}({\mathbf{0}})$. (Recall the definitions of $\mathcal{N}({\mathbf{0}})$ and near-balanced flows from Section \ref{sec:structurelower}.)  

Once again we do most of the work with the Bernoulli distribution $C \sim \Ber(p)$ for $0<p<1$. (Note that we do not now require  $p>1/2$.)  

The proof proceeds from layer to layer by showing (in Lemma \ref{lem:n3}) that, with high probability, a $d^{-3}$-near-balanced flow can be routed across  edge layer $E_m$ which forms a bipartite network $\mathcal{B}_m$. Lemma \ref{lem:main2p} puts together the flows across the layers to find the desired flow from $V_{\ell}(\mathbf{0})$ to $V_{\ell}(\mathbf{1})$. In preparation we need the following two lemmas.

In the proof of Lemma \ref{lem:n3} we may have, for each $m$, small subsets of $V_m$ which do not have some property we would like. The following lemma bounds the impact that these `bad' vertices can have. 

\begin{lemma} \label{lem:bad}
Let ${\ell} \le m \le d-{\ell}$;  let $\{X_i:1\le i\le\binom{d}{m}\}$ be a family of independent random variables taking values 0 or 1 where $\Pro[X_i=1]\le d^{-100}$ for all $i$, and let $S=\sum_i X_i$. Then, for large $d$,
 \[\Pro\left[S \ge \binom{d}{m}d^{-99}\right]\le e^{-d^2}. \]
\end{lemma}

\begin{proof}
By Lemma \ref{ineq:chern}, we have 
$$
\Pro\left[S\ge\binom{d}{m}d^{-99}\right]  \leq \Pro\left[S-\Ex[S] \ge\frac 12 \binom{d}{m}d^{-99}\right] \le e^{-\frac 18 \binom{d}{m} d^{-198}}.
$$
But 
$\binom{d}{m} \ge \binom{d}{\ell} \ge \left(\frac{d}{\ell}\right)^{\ell}=d^{(1-\kappa)d^{\kappa}+O(1)}\ge 8 d^{200}$, for large $d$. 
\end{proof}

The proof of Lemma \ref{lem:n3}  considers the bipartite network of edges across layer $m$ as the network ${\mathcal{B}^*}$ formed by the superposition of two independent networks: ${\mathcal{B}^{\prime}}$ in which edges are present with probability $p^{\prime}$, with $p^{\prime}$ close to but less than $p$ and ${\mathcal{B}_{\delta}}$ in which edges are present with small probability $\delta$. The proof of Lemma \ref{lem:n3}  looks at the flow imbalances that occur at each vertex in the ${\mathcal{B}^{\prime}}$ network if a uniform flow reaches one vertex class $V_{m-1}$, a uniform flow leaves the other vertex class $V_{m}$, and the full capacity of every edge in ${\mathcal{B}^{\prime}}$ connecting the two classes is used. We show that at all except a small number of vertices these imbalances are small. We then use the  network $\mathcal{B}_{\delta}$ to smooth these imbalances so they are very small. We need the following technical lemma to quantify this `smoothing'.

\begin{lemma}\label{lem:smooth}
Fix $\delta, \lambda$ with $1/2 <\lambda<\kappa<1$  and $0<\delta<1$. Fix $m$ such that $d^{\kappa}<m<d/2$ and let $\psi_{max}=2d^{\lambda}|E_{m}|^{-1}$. Let $G_{\delta}$ be the bipartite graph with vertex classes $V_{m-1}$ and  $V_{m}$, in which edges of $Q^d$ are present independently with probability $\delta$; and let $\mathcal{B}_{\delta}$ be the network formed by giving each edge of $G_{\delta}$ capacity $d^{(\lambda-\kappa)/2}|E_m|^{-1} $. Let $V=V_m \cup V_{m-1}$. Let $\{\psi(x):x \in V\}$ be a family of random variables such that the following hold:
\begin{itemize}
	\item the random variables $\left(\psi(x)\right)_{x \in V}$  are independent of the edge-capacities,
	\item the random variables $\left(\psi(x)\right)_{x \in V_{m-1}}$  are independent,
	\item the random variables $\left(\psi(x)\right)_{x \in V_m}$ are independent, 
	\item for each $x \in V$, $|\psi(x)|\le\psi_{max}$,
	\item $\Ex[\psi(x)]=0$ for all $x$. 
\end{itemize}
Then, with failure probability $O(e^{-d^{2}})$, there is a function $\theta:V\rightarrow \mathbb{R}$ and a feasible flow in $\mathcal{B}_{\delta}$  such that the net inflow at each vertex is $\psi(x) +\theta(x)$ and $\sum_{x \in V}{|\theta(x)|}\le d^{-4}$. 
\end{lemma}

\begin{proof}
We denote by $T^s$ ($s$ for small degree) the set of vertices in $G_{\delta}$ with degree less than half their expected values ($\delta m/2$ or $\delta(d-m+1)/2$ for vertices in $V_m$ and $V_{m-1}$ respectively). We start by showing that there are very few vertices in $G_{\delta}$ for which there is a vertex in $T^s$ within a distance 40 in $Q$. Let
\begin{align*}
T^2_m&=\{u \in V_m:\exists v \in T^s \text{ with } d_Q(u,v) \le 40\},\\
T^2_{m-1}&=\{u \in V_{m-1}:\exists v \in T^s \text{ with } d_Q(u,v) \le 40\},\\
T^2 &= T^2_{m-1} \cup T^2_{m}.
\end{align*}
The events $d_{G_{\delta}}(v)<\delta m/2$ for $ v \in V_m$ are independent and, by Lemma \ref{ineq:chern}, $\Pro[d_{G_{\delta}}(v)<\delta m/2]\le e^{-\delta^2m^2/8d}\le e^{-\delta^2d^{2\kappa-1}/8}\le d^{-100}$ for large $d$. Then the distribution of the number of vertices with degree less than half their expected value is stochastically less than $\Bin(|V_m|,d^{-100})$, and so by Lemma \ref{lem:bad}, the number is less than $|V_m|d^{-99}$ with failure probability $O(e^{-d^2})$. Each of these vertices is within a distance 40 of at most $d^{40}$ other vertices so $\Pro[|T_m^2|\ge |V_m|d^{-59}]=O(e^{-d^2})$. A similar argument shows $\Pro[|T_{m-1}^2|\ge |V_{m-1}|d^{-59}]=O(e^{-d^2})$.

Let $\mathcal{B}$ be any bipartite network with vertex classes $V_{m-1}$ and $V_m$ and edge-capacities $d^{(\lambda-\kappa)/2}|E_m|^{-1}$ and also with the properties that $|T_m^2|\le |V_m|d^{-59}$ and $|T_{m-1}^2|\le |V_{m-1}|d^{-59}$. We condition on the network $\mathcal{B}_{\delta}$ being the network~$\mathcal{B}$. 

For each vertex $x\in V\setminus T^2  $ we use the open edges in $\mathcal{B}$ to carry a flow of volume $|\psi(x)|$  to or from (depending on whether $\psi(x)>0$ or $\psi(x)<0$) the vertices $y$ where $d_Q(y,x)=d_{\mathcal{B}}(y,x)=40$. If $\psi(x)>0$ the flows are constructed by pushing a flow of magnitude $\psi(x)$ equally to all the neighbours of $x$ in $\mathcal{B}$ and for each of those vertices then splitting the flow equally between its neighbours in $\mathcal{B}$ a distance one further away from $x$ in both $Q$ and $\mathcal{B}$ and so on. If $\psi(x)<0$ a similar flow of volume $-\psi(x)$ is constructed in the opposite direction. We shall see that, since $x \in V \setminus T^2$, there are always many neighbours one further away from $x$ in both $Q$ and $\mathcal{B}$.

We denote by $f_x(z)$ the amount of flow that reaches (or comes from) a destination vertex $z$ from a source (or sink) $x$. ($f_x(z)$ has the same sign as $\psi(x)$). We denote by $X(z)$ the set of vertices that might send a flow to a destination $z$ (or receive from $z$ if $\psi(x)<0$) and by $F(z)=\sum_{x \in X(z)}f_x(z)$, the resultant flow at $z$. 

Two vertices $x,z$ in $V_m$ (or equally $V_{m-1}$) are separated by a distance 40 (in $Q$) if 20 of the elements of (the set representing) $x$ are replaced by 20 elements of $[d] \setminus x$ to obtain (the set representing) $z$. For each $x \in V_m$ there are therefore $\binom{m}{20}\binom{d-m}{20}$ such vertices, and for each $x \in V_{m-1}$ there are $\binom{m-1}{20}\binom{d-m+1}{20}$ such vertices. 

The magnitude $|f_x(z)|$  depends on the size of the initial imbalance $|\psi(x)|$, the number of paths from $x$ to $z$ that exist in $\mathcal{B}$ and the `onward vertex degree' (i.e. the number of edges in $\mathcal{B}$ leaving a vertex on the path to a vertex further away in $Q$ and $\mathcal{B}$ from $x$) of each vertex on these paths. For a given $x$ and $z$ there are $(20!)^2$ possible paths from $x$ to $z$ corresponding to the different possible orders of removing and adding elements from $x$ to reach $z$. For a particular path $\mathcal{P}$ we define   $\nu(\mathcal{P})=1$ if all edges in $\mathcal{P}$ are open and $\nu(\mathcal{P})=0$ otherwise and we denote by $c(\mathcal{P})$ the product of the onward degrees along $\mathcal{P}$. Then  $f_x(z)$ can be expressed as
\[f_x(z)=\psi(x)\sum_{\mathcal{P}}  \frac{\nu(\mathcal{P})}{c(\mathcal{P})},\]
where the sum is taken over all $(20!)^2$ paths of length 40 from $x$ to $z$. We make two observations about this expression. Firstly, the sum $\sum_{\mathcal{P}}  \frac{\nu(\mathcal{P})}{c(\mathcal{P})}$ depends only on which edges in $\mathcal{B}$ are open and is fixed as we are conditioning on $\mathcal{B}$. For a given $z$ the variables $\{f_x(z):x \in V_m\}$ are therefore independent (though not, in general, with the same distribution). Secondly, each vertex in the path has degree at least half its expected degree (either $\delta m/2$ or $\delta (d-m+1)/2$). Therefore,  
\[
c(\mathcal{P})  \ge \prod_{i=1}^{20}\left(\frac{m\delta}{2}-i\right)\left(\frac{(d-m+1)\delta}{2}-i\right),\]
and so
\[ m^{20}(d-m+1)^{20}  \ge c(\mathcal{P}) \ge m^{20}(d-m+1)^{20}(\delta/3)^{40}.\]
Since $d/2 \le d-m+1 \le d$,
\[|f_x(z)|\le \frac{c\psi_{max}}{(dm)^{20}},\]
where  $c=(20!)^2(3/\delta)^{40}2^{20}$.

 The expressions within the sum $F(z)=\sum_{x \in X(z)}f_x(z)$ are independent, bounded and have zero expectation, so by the Chernoff bounds (Lemma \ref{ineq:chern}),
\[\Pro \Big[|F(z)|>a\frac{ c \psi_{max}}{(md)^{20}}\Big] < 2e^{-a^2/2|X(z)|}. \]
We note  $|X(z)| \le \binom{d}{20}^2 \le d^{40}/2$ and put $a=d^{20+4/3}$ to get $\frac{a^2}{2|X(z)|}\ge d^{8/3}$. Since $\psi_{max}=2d^{\lambda} |E_m|^{-1}$ and $(md)^{20}\ge d^{(1+\kappa)20}$, we get 
\[\Pro \Big[|F(z)|>(|E_m|d^5)^{-1}\Big] =O(e^{-2d^2}). \]
This result holds for each $z \in V$ and so
\[\Pro\Big[|F(z)|>(|E_m|d^5)^{-1}\text{ for some }z \in V\Big] =O(e^{-d^2}).\]

After this process we have the original imbalances $\psi(z)$ at vertices in $T^2$ which were not `smoothed', and the imbalances $F(z)$ arising from the smoothing process. We now introduce an equal and opposite flow $\theta(z)$ into every vertex in $\mathcal{B}$  to make the overall flow feasible. For $z \in T^2$, we set $\theta (z) = -\psi(z)-F(z)$ and for $z \in V \setminus  T^2$, we set $\theta (z) = -F(z)$. So

\[\sum_{z \in V} |\theta(z)| \le  \sum_{z \in V} |F(z)|+\sum_{z \in T^2}|\psi(z)|.\]

For $z \in T^2_{m}$, we have $|\psi(z)| \le \psi_{max}\le |V_m|^{-1}$ for large $d$; and similarly, for $z \in T^2_{m-1}$, we have $|\psi(z)| \le |V_{m-1}|^{-1}$. So then
\[\sum_{z \in T^2}|\psi(z)|\le |T^2_{m}||V_{m}|^{-1}+ |T^2_{m-1}||V_{m-1}|^{-1}\le 2d^{-59}.\]

We also have $\sum_{z \in V_m }|F(z)| \le |V_{m}||E_m|^{-1}d^{-5}\le d^{-5}$ with failure probability $O(e^{-d^2})$ and similarly $\sum_{z \in V_{m-1}}|F(z)| \le |V_{m-1}||E_m|^{-1}d^{-5}\le d^{-5}$ with failure probability $O(e^{-d^{2}})$. So with failure probability  $O(e^{-d^2})$,
\[\sum_{x \in V} |\theta(x)| \le d^{-4}.\]

We must check that these flows do not exceed the capacities of  $\mathcal{B}$. Consider edge $e=uv$ where $u \in V_{m-1}$ and $v \in V_{m}$. We denote by $F(e)=\sum_{\mathcal{P}}f_x(z)$ the total flow in edge $e$ of all the $f_x(z)$ flows for all possible $x$ and for all possible $z$ where $e$ is in a path in $Q[V]$ of length 40 from $x$ to $z$. The total possible number of paths of length 40 is $|V_{m}|\binom{m}{20}\binom{d-m}{20}+|V_{m-1}|\binom{m-1}{20}\binom{d-m+1}{20}$ and so, by the symmetry of the cube, each edge is in at most 40 times this number divided by the total number of edges (i.e. $|E_m|=|V_{m}|m=|V_{m-1}|(d-m+1)$). So each edge is in 
\[\frac{40}{m}\binom{m}{20}\binom{d-m}{20}+\frac{40}{d-m+1}\binom{m-1}{20}\binom{d-m+1}{20}=O\left(\frac{(md)^{20}}{m}\right)\]
paths. So $|F(e)|$ is at most this number of paths times the maximum value of any individual $|f_x(z)|$ i.e.
\[|F(e)| =O\left(\frac{(md)^{20}}{m} \frac{\psi_{max}}{(md)^{20}}\right)=O\left(\frac{\psi_{max}}{m}\right)=O\left(d^{\lambda-\kappa}|E_m|^{-1}\right).
\]
The capacity of the open edges is $d^{(\lambda-\kappa)/2}|E_m|^{-1}$ so the ratio of the flow calculated above to the edge-capacity is $O(d^{(\lambda-\kappa)/2})$. So, for large enough $d$,  the smoothing can be accomplished within $\mathcal{B}$.

We have thus shown that by conditioning on $\mathcal{B}_{\delta}$ being a specific network $\mathcal{B}$ with the properties given (essentially that $T^2$ is small) we achieve the required flows with failure probability $O(e^{-d^{2}})$ and we have also shown that $\mathcal{B}_{\delta}$ has these properties with failure probability $O(e^{-d^{2}})$. Thus we can achieve the flows we require in $\mathcal{B}_{\delta}$ with failure probability $O(e^{-d^{2}})$.
\end{proof}
We now prove the lemma that  gives us a near balanced flow across a single layer. We consider a bipartite network with vertex classes corresponding to $V_{m-1}$ and $V_m$ where each edge of $Q^d$ is present independently with probability $p$.  Edges are given capacity $(1+\epsilon)|E_m|^{-1}$ so that, with high probability, the total capacity is close to $(1+\epsilon)p$. We then introduce a balanced flow of volume $p$ into $V_{m-1}$ and extract a balanced flow of volume $p$ from $V_m$ and show that, with high probability, the resultant flow in the network is feasible except for very small imbalances.

\medskip
\begin{lemma}\label{lem:n3}
Fix $\epsilon>0$ and suppose that  $0<p<1$  and $\ell+1 \le m \le d-\ell$. Let  $G_m$ be the bipartite subgraph induced in $Q^d$ with vertex classes $V_{m-1}$ and $V_{m}$, and let $\mathcal{B}_m$ be the network formed by picking edges of $G_m$ independently with probability $p$ and giving them capacity $(1+\epsilon)|E_m|^{-1}=(1+\epsilon)\left(m\binom{d}{m}\right)^{-1}$. Then,  with failure probability $O(e^{-d^2})$, there exists a  $(d^{-3}/9)$-near-balanced feasible flow of volume $p$ from  $V_{m-1}$ to $V_{m}$ in $\mathcal{B}_m$. 
\end{lemma}
\begin{proof}
Choose $p'$ with $\max\left\{\frac p{1+p},\frac p{1+\epsilon}\right\}\le p'<p$ and define $\delta$ by $(1-p)=(1-p')(1-\delta)$. Choose $\lambda$ with $1/2<\lambda<\kappa$ . We then define the networks $\mathcal{B}^{\prime}$, $\mathcal{B}_{\delta}$ and $\mathcal{B}^*$ as follows. $\mathcal{B}^{\prime}$ is formed by picking edges of $G_m$ independently with probability $p^{\prime}$ and giving them capacity $\frac{p}{p^{\prime}}|E_m|^{-1}$. $\mathcal{B}_{\delta}$ is formed by picking edges of $G_m$ independently with probability $\delta$ and giving them capacity $d^{(\lambda-\kappa)/2}|E_m|^{-1}$. Next, we let $\mathcal{B}^*$ be the network formed by superposing $\mathcal{B}^{\prime}$ and $\mathcal{B}_{\delta}$. Finally we let $\mathcal{B}$ be the network formed by giving an edge-capacity $(1+\epsilon)|E_m|^{-1}$ if the corresponding edge in $\mathcal{B}^*$ has non-zero capacity and capacity zero otherwise. We note that $\mathcal{B}$ has the same distribution as $\mathcal{B}_m$ and the capacity of any edge in $\mathcal{B}$ is at least the capacity of the corresponding edge in $\mathcal{B}^*$. So to show that flows exist in $\mathcal{B}_m$ we can analyse flows in $\mathcal{B}^*$. 

We define random sets
\begin{align*}
&T^{\prime}_m=\{u \in V_m:d'(u)<p'm-d^{\lambda}  \text{ or } d'(v)>p'm+d^{\lambda}\},\\
&S^{\prime}_m=V_m\setminus T^1_m,
\end{align*}
where $d'(v)$ is the degree of $v$ in the $p'$ network $\mathcal{B}^{\prime}$. We similarly define $T^{\prime}_{m-1}$ and $S^{\prime}_{m-1}$  for vertices in layer $V_{m-1}$. 

We now condition on the values of the random sets $T^{\prime}_{m-1}$ and $T^{\prime}_m$. So let us pick subsets $T_{m-1} \subset V_{m-1}$ with $|T_{m-1}|\le|V_{m-1}|d^{-99}$, and $T_{m} \subset V_{m}$ with $|T_{m}|\le|V_{m}|d^{-99}$ and condition on the events $T^{\prime}_{m-1}=T_{m-1}$  and $T^{\prime}_{m}=T_{m}$. 

We introduce a balanced inflow to $V_{m-1}$ of volume $p$ and demand a balanced outflow from $V_m$ of volume $p$, and now define a flow  from $V_{m-1}$ to $V_{m}$ as follows.
\begin{enumerate}
	\item Use the full capacity of every edge in $\mathcal{B}^{\prime}$ to route flow from $V_{m-1}$ to $V_{m}$.
	\item For $u \in V_m$ the imbalance $\rho(u)$ is given by
	\[
	\rho(u)=\frac{p}{p'} d'(u)|E_{m}|^{-1}-p|V_m|^{-1},
	\]
	and we make two observations about $\rho(u)$. Firstly,
	\[
	\frac{p}{p^{\prime}}mp^{\prime}|E_m|^{-1}=pm|E_m|^{-1}=p|V_m|^{-1},~ \text{so} ~\Ex[\rho(u)]=0.
	\]
	Secondly $0 \le d'(u) \le m$,~ so 
	\[
	-p|V_m|^{-1}\le\rho(x)\le \frac{p}{p^{\prime}}m|E_m|^{-1}-p|V_m|^{-1}=\left(\frac{p}{p^{\prime}}-p\right)|V_m|^{-1},
	\] 
	Since $p/p^{\prime}\le 1+p$ we have $|\rho(x)| \le |V_m|^{-1}$. 
	
	For $u \in V_{m-1}$ the imbalance $\rho(u)$ is given by
	\[
	\rho(u)=-\frac{p}{p'} d'(u)|E_{m-1}|^{-1}+p|V_{m-1}|^{-1},
	\]
	and we make the same two observations as above. Namely, $\Ex[\rho(u)]=0$ and $|\rho(u)| \le |V_{m-1}|^{-1}$.

			\item Let $S_m =V_m \setminus T_m$ and $S_{m-1} =V_{m-1} \setminus T_{m-1}$. We define $\psi(x)$ as follows: 
			\begin{equation}
			\psi(x)=\begin{cases} \rho(x)-|S_m|^{-1}\sum_{v \in S_m}\rho(v)&\mbox{if } x \in S_m\\
			                      0                                        &\mbox{if } x \in T_m\\
			                      \rho(x)-|S_{m-1}|^{-1}\sum_{v \in S_{m-1}}\rho(v)&\mbox{if } x \in S_{m-1}\\
			                      0                                        &\mbox{if } x \in T_{m-1}.
			                      \end{cases}
			                      \end{equation}
			For each vertex $x$, flows of volume $\psi(x)$ are now put into $\mathcal{B}_{\delta}$ and `smoothed' using Lemma \ref{lem:smooth}. The quantities $\psi(x)$ and the network $\mathcal{B}_{\delta}$ satisfy the conditions of Lemma \ref{lem:smooth} so the smoothing process takes the input $\psi(x)$ and results in output $\theta(x)$, and the resulting imbalance after smoothing at each vertex is $\phi(x)=\rho(x)-\psi(x)+\theta(x)$.
			\end{enumerate}
So

\begin{align}\label{eqn:sumphi}
\sum_{x \in V} |\phi(x)| &\leq \sum_{x \in V} |\rho(x)-\psi(x)| + \sum_{x \in V} |\theta(x)|,\nonumber\\
&\leq \sum_{x \in T_m} |\rho(x)-\psi(x)| +\sum_{x  \in T_{m-1}} |\rho(x)-\psi(x)|\nonumber\\
&+\sum_{x \in S_m } |\rho(x)-\psi(x)|+\sum_{x \in S_{m-1}} |\rho(x)-\psi(x)|+ \sum_{x \in V} |\theta(x)|.
\end{align}
We are able to achieve a small total for this expression because $|\rho(x)-\psi(x)|$ is small on $T_m \cup T_{m-1}$ which is a small set and because $|\rho(x)-\psi(x)|$ is very small on $S_m \cup S_{m-1}$. Formally for $x \in T_m, \psi(x)=0$ and $|\rho(x)|\le|V_{m}|^{-1}$; and for $x \in T_{m-1}, \psi(x)=0$ and $|\rho(x)|\le|V_{m-1}|^{-1}$. For $x\in S_{m}, \rho(x)-\psi(x)=|S_{m}|^{-1}\sum_{v \in S_{m}}\rho(v)$; and for $x \in S_{m-1}, \rho(x)-\psi(x)=|S_{m-1}|^{-1}\sum_{v \in S_{m-1}}\rho(v)$.

So by equation(\ref{eqn:sumphi})
\begin{align}
\sum_{x \in V} |\phi(x)| &\leq \frac{|T_m|}{|V_m|} +\frac{|T_{m-1}|}{|V_{m-1}|}
&+\left|\sum_{v \in S_m}\rho(v)\right|+\left|\sum_{v \in S_{m-1}}\rho(v)\right|+ \sum_{x \in V} |\theta(x)|.\label{eq:smoothterms}
\end{align}
We now look at each of these terms in turn.
By our conditioning \newline
$|T_m||V_m|^{-1}<d^{-99}$ and $|T_{m-1}||V_{m-1}|^{-1}<d^{-99}$, and
\begin{align*}
\left|\sum_{v \in S_m}\rho(v)\right|&=\left|\sum_{v \in V_m}\rho(v)-\sum_{v \in T_m}\rho(v)\right|,\\
&\le\left|\sum_{v \in V_m}\rho(v)\right|+ |T_m||V_m|^{-1},\\
&\le \left|\sum_{v \in V_m}\rho(v)\right|+d^{-99}.
\end{align*}
From Lemma \ref{ineq:chern} we have
\[\Pro\left[\left|\sum_{v \in V_m}\rho(v)\right|>d^{-99}p\right]\le2 e^{-d^{-198}p^2\binom{d}{m}/2}=O(e^{-d^2}),\]and
\[\Pro\left[\left|\sum_{v \in V_{m-1}}\rho(v)\right|>d^{-99}p\right]\le2 e^{-d^{-198}p^2\binom{d}{m-1}/2}=O(e^{-d^2}).\]
From Lemma \ref{lem:smooth} we get
\[\Pro\left[\sum_{v \in V_m}\left|\theta(v)\right|>d^{-4}\right]=O\left(e^{-d^2}\right).\]
Putting these expressions back into equation (\ref{eq:smoothterms}), subject to the initial conditioning, with failure probability $O(e^{-d^2})$, we get that $\sum_{x \in V} |\phi(x)|<2d^{-4}$, which is at most $d^{-3}p/7$, for large $d$.

It now remains to show that, with high probability, our network $\mathcal{B}$  satisfies $|T^1_m| \le |V_m|d^{-99}$ and $|T^1_{m-1}| \le |V_{m-1}|d^{-99}$. 

For $x \in V_m$, $d^{\prime}(x)$ is a random variable with distribution $\Bin(m,p^{\prime})$ so by taking Chernoff bounds (Lemma \ref{ineq:chern}) $\Pro[x \in T^1_m]<e^{-d^{2\lambda}/2m}=e^{-\Omega(d^{2\lambda-1})}$. The events $x \in T^1_m$ for $x \in V_m$ are independent so we can use Lemma \ref{lem:bad} to show $\Pro[|T^1_m| \ge |V_m|d^{-99}]=O(e^{-d^2})$. So with failure probability $O(e^{-d^2})$, $|T^1_m||V_m|^{-1}\le d^{-99}$. Similarly, with failure probability $O(e^{-d^2})$, we have \newline  \mbox{$|T^1_{m-1}||V_{m-1}|^{-1}\le d^{-99}$.} This proves the lemma.

We note that we require $\lambda>1/2$ and hence $\kappa>1/2$ for this proof to work. In essence the proof would not work if the disparity between the sizes of the two vertex layers was too large.
\end{proof}

Next we are going to prove a general result about flows in directed networks. We will use the definitions of a proper $ST$ flow $f$ and its volume $\vol(f)$ from Section \ref{3sec:definitions}.

\begin{lemma}\label{lem:stitching}
Let $\mathcal{N}=(V,\overrightarrow{E})$ be a directed network and let $S$ and $T$ be disjoint non-empty subsets of $V$. Let $0<\theta<1/9$ and let $f:\overrightarrow{E}\rightarrow [0,\infty)$ be a possibly improper flow with the following properties:
\begin{align*}
&\sum_{v \in V \setminus (S \cup T)}|f^+(v)|\le \theta\, \size(f),\\
&\sum_{v \in S}\left|f^+(v)-\size(f)/|S|\right|+\sum_{v \in T}\left|f^-(v)-\size(f)/|T|\right|\le \theta \,\size(f),
\end{align*}
where 
$\size(f)=1/2\sum_{v \in V}|f^+(v)|.$
Then there is a proper $ST$ flow $g$ 
such that the following hold:
\begin{enumerate}
	\item For all $\overrightarrow{e} \in \overrightarrow{E}$, $0\le g(\overrightarrow{e}) \le f(\overrightarrow{e})$.
	\item For all $s \in S$, $g^+(s)\ge 0$ and for all  $t \in T$, $g^-(t) \ge 0$, and $\vol(g)=\sum_{s \in S}g^+(s)=\sum_{t \in T} g^-(t)\ge (1-2\theta)\,\size(f).$
	\item $\sum_{v \in S}\left|g^+(v)-\vol(g)/|S|\right|+\sum_{v \in T}\left|g^-(v)-\vol(g)/|T|\right|\le 9\,\theta\, \vol(g).$
\end{enumerate}
\end{lemma}
\begin{proof}
We define
\begin{align*}
V^+&=\{v \in V \setminus (S \cup T):f^+(v) >0\},\\
V^-&=\{v \in V \setminus (S \cup T):f^+(v) <0\},\\
S^+=&\{s \in S:f^+(s) > 0\}\text{ and }S^-=\{s \in S:f^+(s) < 0\},\\
T^+=&\{t \in T:f^+(t) > 0\}\text{ and }T^-=\{t \in T:f^+(t) < 0\}.
\end{align*}
And we introduce a `super-source' vertex $x$, and edges $\overrightarrow{E}^+=\{\overrightarrow{xv}: v \in V^+ \}$, $\overrightarrow{E}^+_S=\{\overrightarrow{xs}: s \in S^+ \}$,$\overrightarrow{E}^+_T=\{\overrightarrow{xt}: t \in T^+ \}$  and a `super-sink' vertex $y$ and edges $\overrightarrow{E}^-=\{\overrightarrow{vy}: v \in V^- \}$, $\overrightarrow{E}^-_S=\{\overrightarrow{sy}: s \in S^- \}$ and $\overrightarrow{E}^-_T=\{\overrightarrow{ty}: t \in T^- \}$. (see Figure \ref{fig14}). We now define $f^{\prime}$ by
\[
f^{\prime}(\overrightarrow{e})=
\begin{cases}
f(\overrightarrow{e}) & \text{if } \overrightarrow{e} \in \overrightarrow{E},\\
f^{+}(v) & \text{for } \overrightarrow{e}=xv \in \overrightarrow{E}^+ \cup \overrightarrow{E}^+_S \cup \overrightarrow{E}^+_T,\\
f^{-}(v) & \text{for } \overrightarrow{e}=vy \in \overrightarrow{E}^- \cup \overrightarrow{E}^-_S \cup \overrightarrow{E}^-_T.
\end{cases}
\]
The new flow $f^{\prime}$ is a proper $xy$ flow on the enlarged network. Also $\vol(f^{\prime})=\sum\{f^+(v):v \in V \text{ with } f^+(v)>0\}=\size(f)$. We can decompose this flow into flows along $xy$ paths and around cycles (see for example Ahuja, Magnanti and Orlin \cite[page 80]{AhujaMagnantiOrlin}). We then define a new proper $xy$ flow by deleting all the $xy$ flows along paths in this decomposition that use edges in $\overrightarrow{E}^+ \cup \overrightarrow{E}^+_T \cup\overrightarrow{E}^- \cup\overrightarrow{E}^-_S $. The sum of the volumes of the deleted $xy$ flows is at most 
\[\sum_{v \in V\setminus (S \cup T)} |f^+(v)|+\sum_{v \in S^-}|f^-(v)|+\sum_{v \in T^+}|f^+(v)|.\]
Now these last two terms sum to less than $\sum_{v \in S}|f^+(v)-\size(f)/|S||+\sum_{v \in T}|f^-(v)\minus\size(f)/|T||$ which is less than $\theta\,\size(f)$.

So the total volume of the deleted flows is at most $2\theta\,\size(f)$ and the resultant flow is a proper $xy$ flow with volume at least $(1-2\theta) \,\size(f)$ in which every constituent flow starts $xs$ and ends $ty$ for some vertices $s \in S, t \in T$. We now define our proper $ST$ flow $g$ by restricting this flow to the edges in $\mathcal{N}$, i.e. the $st$ portions of these paths.

By construction $0\le g(\overrightarrow{e}) \le f(\overrightarrow{e})$ for all $\overrightarrow{e} \in \overrightarrow{E}$ and $g^+(s)\ge 0$ for all $s \in S$ and $g^-(t)\ge 0$ for all $t \in T$ and we have just shown $\vol(g)=\sum_{s \in S}g^+(s)=\sum_{t \in T}g^-(t)\ge (1-2\,\theta)\,\size(f)$. Lastly,
\begin{align*}
\sum_{v \in S}\left|g^+(v)-\vol(g)/|S|\right| &\le \sum_{v \in S}\left|g^+(v)-f^+(v)\right|+\sum_{v \in S}\left|f^+(v)-\size(f)/|S|\right|\\
&+\sum_{v \in S}\left|\size(f)/|S|-\vol(g)/|S|\right|.
\end{align*}
And the last term equals $|\size(f)-\vol(g)|\le 2\theta\,\size(f)$. So
\begin{align*}
\sum_{v \in S}\left|g^+(v)-\vol(g)/|S|\right|&+\sum_{v \in T}\left|g^-(v)-\vol(g)/|T|\right| \le\sum_{v \in S\cup T} |f^+(v)-g^+(v)|\\
&+\sum_{v \in S}\left|f^+(v)-\size(f)/|S|\right|+\sum_{v \in T}\left|f^-(v)-\size(f)/|T|\right|\\
&+|S|^{-1}\sum_{v \in S}\left|\size(f)-\vol(g)\right|+|T|^{-1}\sum_{v \in T}\left|\size(f)-\vol(g)\right|,\\
&\le 7 \theta \,\size(f) \le 9\theta\,\vol(g).
\end{align*}
\end{proof}

\floatstyle{plain}
\restylefloat{figure}
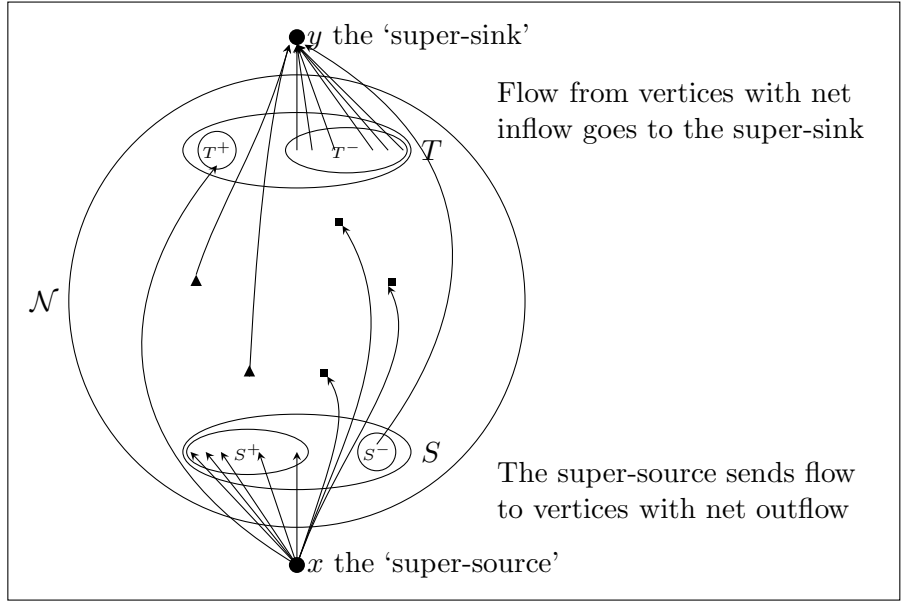
\begin{figure}[h]
	\centering
	\small
\begin{tikzpicture}[show background rectangle, >=stealth]
\draw (0,3) circle (3);
\draw (-3,3) node[left]{$\mathcal{N}$};
\draw (0,1) ellipse (1.5cm and 0.5cm);
\draw (-0.65,1) ellipse (0.8cm and 0.3cm);
\draw (-0.65,1) node {\tiny$S^+$\normalsize};
\draw (1.05,1) circle (0.25cm);
\draw[->] (1.05,1.1) .. controls (2.5,3) and (2.5,5).. (0.1,6.4);
\draw (1.05,1) node {\tiny$S^-$\normalsize};
\draw (0,5) ellipse (1.5cm and 0.5cm);
\draw (0.65,5) ellipse (0.8cm and 0.3cm);
\draw (0.65,5) node {\tiny$T^-$\normalsize};
\draw[->] (0,-0.5) .. controls (-2.5,1) and (-2.5,3).. (-1.05,4.8);
\draw (-1.05,5) circle (0.25cm);
\draw (-1.05,5) node {\tiny$T^+$\normalsize};
\draw (1.5,1) node[right]{$S$};
\draw (1.5,5) node[right]{$T$};
\fill[black] (0,-0.5) circle(3pt) node[right]{$x$ the `super-source'};
\fill[black] (0,6.5) circle(3pt) node[right]{$y$ the `super-sink'};
\draw [->](0,-0.5) -- (-1.4,1);
\draw [->](0,-0.5) -- (-1.2,1);
\draw [->](0,-0.5) -- (-1.0,1);
\draw [->](0,-0.5) -- (-0.5,1);
\draw [->](0,-0.5) -- (0,1);
\draw [->](0,5) -- (0,6.4);
\draw [->](0.2,5) -- (0,6.4);
\draw [->](0.5,5) -- (0,6.4);
\draw [->](1,5) -- (0,6.4);
\draw [->](1.2,5) -- (0,6.4);
\draw [->](1.4,5) -- (0,6.4);
\fill[black] (0.3,2) rectangle(0.4,2.1);
\fill[black] (0.5,4.0) rectangle(0.6,4.1);
\fill[black] (1.3,3.2) rectangle(1.2,3.3);
\fill[black] (-0.7,2) -- (-0.625,2.15) -- (-0.55,2) -- cycle;
\fill[black] (-1.4,3.2) -- (-1.325,3.35) -- (-1.25,3.2) -- cycle;
\draw[->] (0,-0.5) .. controls  (0.5,1) and (0.7,1.5).. (0.4,2);
\draw[->] (0,-0.5) .. controls  (0.5,1) and (1.5,2.5).. (0.6,4.0);
\draw[->] (0,-0.5) .. controls  (0.5,1) and (1.7,2.2).. (1.2,3.2);
\draw[->] (-0.625,2) .. controls (-0.6,2.35) and (-0.5,5).. (-0.1,6.4);
\draw[->] (-1.325,3.35) .. controls (-1,4.35) and (-0.5,5).. (-0.1,6.4);
\draw (2.5,5.5) node[right, text width=5cm]
{Flow from vertices with net inflow goes to the super-sink };
\draw (2.5,0.5) node[right, text width=5cm]
{The super-source sends flow to vertices with net outflow };
\end{tikzpicture}
\normalsize
\caption{\mdseries Flows added from a super-source and to a super-sink to create a proper flow. \normalsize}
	\label{fig14}
\end{figure}
\floatstyle{boxed}
\restylefloat{figure}	
\vspace{7pt}

\begin{lemma}\label{lem:main2p}
Fix $0<p<1$, and fix $\epsilon>0$. For $\mathcal{N} \in \mathcal{G}(Q^d,p)$, with failure probability $e^{-\Omega(d^2)}$ there exists a $ d^{-2}$-near-balanced flow of volume at least $(1+\epsilon)^{-1}p$ from  $V_{\ell}(\mathbf{0})$ to $V_{\ell}(\mathbf{1})$ in the network $\mathcal{N}({\mathbf{0}})$. 
\end{lemma}

\begin{proof}
The proof follows by applying Lemma \ref{lem:n3} to each layer from layer $\ell+1$ to layer $d-\ell$. With failure probability $O(e^{-d^2})$ there is in each layer a $(d^{-3}/9)$-near-balanced flow of volume $p$ that satisfies the capacity constraints. By taking the union bound we see that with failure probability $e^{-\Omega(d^2)}$, there is an improper flow $F$ of volume $p$ from $V_{\ell}$ to $V_{d-\ell}$ with $\sum|f^+(v)| \le d^{-2}p/9$ where the sum is taken over all vertices in layers $m$ to $ d-\ell-1$, and which satisfies the capacity constraints of $\mathcal{N}$. By applying Lemma \ref{lem:stitching} there is a feasible flow of volume $(1-d^{-2})p$ from $V_{\ell}$ to $V_{d-\ell}$, which is $d^{-2}$-near-balanced.
\end{proof}

The main result of this section is the following lemma which is the generalization of Lemma \ref{lem:main2p} to general edge-capacities.

\begin{lemma}\label{lem:main2}
Let the non-negative random variable $C$ have finite mean. Fix $\epsilon>0$. For $\mathcal{N} \in \mathcal{G}(Q^d,C)$, with failure probability $e^{-\Omega(d^2)}$, the network $\mathcal{N}({\mathbf{0}})$ has a 
$d^{-2}$-near-balanced flow of volume at least $(1-\epsilon)\Ex[C]$ from  $V_{\ell}(\mathbf{0})$ to $V_{\ell}(\mathbf{1})$.
\end{lemma}

\begin{proof}
Lemma \ref{lem:main2p} gave the result we desire in a network $\mathcal{N} \in \mathcal{G}(Q^d,p)$. We now look at  $\mathcal{N} \in \mathcal{G}(Q^d,C)$ for the general distribution $C$. We say a random variable $C$ and its distribution are \emph{good} if, with failure probability $e^{-\Omega(d^2)}$, there exists a $d^{-2}$-near-balanced  flow of volume at least $(1-\epsilon)\Ex[C]$ from  $V_{\ell}({\mathbf{0}})$ to $V_{\ell}({{\mathbf{1}}})$ in network $\mathcal{N}(\bf{0})$ where the independent edge-capacities have distribution $C$. So we wish to prove that $C$ is good for general distribution $C$ with $\Ex[C]<\infty$.

Firstly consider a distribution of the form 
\begin{equation}\label{eq:multi}
\Pro[C=a_i]=p_i,\; 0\le i \le n,\;\sum_{i=0}^n p_i=1, \; 0=a_0<a_1<a_2\dots<a_n, 
\end{equation}
where $p_i>0 \text{ for }i>0$. 
Now define $C_i=\mathbf{1}_{C=a_i}$ and note that $C=\sum_{i=0}^n a_iC_i$ and $C_i \sim \Ber(p_i)$.
Now consider a network $\mathcal{N}'$ generated by superimposing networks formed with edge-capacities $a_iC_i$ and note that $\mathcal{N}'$ is a valid generator of $\mathcal{N}$. Now a scaled distribution of a good distribution is good and a finite sum of (not necessarily independent) good distributions is also good. To see this, as $\epsilon >0$, we know, with failure probability $e^{-\Omega(d^2)}$, there exists a $d^{-2}$-near-balanced flow of volume at least $(1-\epsilon)\Ex[a_iC_i]$ in the network formed with edge-capacities $a_iC_i$. By taking the union bound we know that the sum of these flows gives a flow of volume at least $(1-\epsilon) \Ex[\sum_i a_iC_i]$ with failure probability $e^{-\Omega(d^2)}$. The volume of the sum of these flows equals the sum of the volumes of the individual flows so the resultant flow is also $d^{-2}$-near-balanced.

Hence  all $C$ of the  form (\ref{eq:multi}) are good.
For general distribution $C$ with $\Ex[C]<\infty$ and $\epsilon>0$ we can truncate and approximate $C$ by a random variable $C^{(\epsilon)}$ of the form (\ref{eq:multi}) such that $\Ex[C^{(\epsilon)}]>\Ex[C]-\epsilon$ and $C \ge C^{(\epsilon)}$.\end{proof}

%

\section{Proofs of Lemma \ref{lem:0} and Theorem 1}\label{sec:th1proof}
\subsection{Proof of Lemma \ref{lem:0}}

In Lemma \ref{lem:main1} we showed that given $\epsilon>0$ there exists a constant $M$ such that, with failure probability $O(2^{-\rho d})$ we can route a balanced flow of volume $\Ex[C]$ in the network $\mathcal{N}^M(u)$ between $u$ and layer $V_{\ell}(u)$ for each vertex $u$ in $Q^d$. Let us condition on this happening and consider vertex $\bf{0}$. In Lemma \ref{lem:main2} we showed that with failure probability $e^{-\Omega(d^2)}$ we could route a $d^{-2}$-near-balanced flow of volume at least $(1+\epsilon)^{-1}\Ex[C]$ from $V_{\ell}(\bf{0})$ to $V_{d-{\ell}}(\bf{0})$. By Lemma \ref{lem:stitching} these can be combined to give a flow from $\mathbf{0}$ to $\mathbf{1}$ of volume at least $(1+\epsilon)^{-1}\Ex[C]-d^{-2}$. The edges sets used by the three sets of flows defined above only overlap in $E_{\ell+1},\;E_{\ell+2},\;E_{d-\ell-1},\;E_{d-\ell}$  and the sum of the capacities used is less than the capacity in $\mathcal{N}^{M+1}(\mathbf{0})$. Thus for large enough $d$ and by a suitable redefinition of $\epsilon$ we see that subject to our conditioning with failure probability $e^{-\Omega(d^2)}$ we can route a balanced flow of volume $(1-\epsilon)\Ex[C]$ in the network $\mathcal{N}^{M+1}(\mathbf{0})$ between $\mathbf{0}$ and  $\mathbf{1}$. We can remove the condition to prove the lemma. \qed


\subsection{Proof of  Theorem 1}\label{sec:proof1}
The upper bound (\ref{ub}) has already been proved in Section \ref{sub:upper} so it remains to prove the lower bound (\ref{lb}).
Lemma \ref{lem:0} showed that given $\epsilon>0$ there exists a constant $M$ such that, with failure probability $O(2^{-\rho d})$, for each antipodal vertex pair $u,\overline{u}$, there exists a flow of volume  $(1-\epsilon)\Ex[C]$ between  $u$ and $\overline{u}$ in the network $\mathcal{N}^M(u)$. We now consider the flows of all antipodal vertex pairs simultaneously in the network $\mathcal{N}_{sum}$ which is formed by superposing the networks $\mathcal{N}^M(u)$ for all vertex pairs. The lower bound in Theorem \ref{th:1} is proved by showing in Lemma \ref{lem:qsumcap} that the capacity demanded of an edge $e$ in $\mathcal{N}_{sum}$ is less than $(1+o(1))c_e$ and then rescaling.

The capacity $\capt(e)$ of edge $e$ required in $\cN_{sum}$ is the sum of the capacity required in that edge for all antipodal vertex pairs. Then
\[\capt(e)=\frac12 \sum_{v \in V} \mathbf{1}_{e \in E_m(v)} \frac{c_eM(m)}{|E_m|},\]
where $M(m)=M$ for $1\le m\le\ell+2$ and for $d-\ell-1\le m\le d$ and $M(m)=1$ otherwise, and $\mathbf{1}_{e \in E_m(v)}$ is an indicator function that takes the value 1 when $v$ is a vertex for which the edge $e$ is in edge-layer $m$. For ease of calculation we write this out as follows.
\begin{align*}
\capt(e)=\sum_{m=1}^d \sum_{v:e \in E_m(v)} \!\frac{c_e}{2|E_m|}&+\sum_{m=1}^{\ell+2} \sum_{v:e \in E_m(v)} \! \frac{c_e (M\minus 1)}{2|E_m|}\\
&+\sum_{m=d-\ell-1}^d \sum_{v:e \in E_m(v)} \! \frac{c_e (M\minus 1)}{2|E_m|}.
\end{align*}

\begin{lemma}\label{lem:qsumcap}
Fix an edge $e$ in $Q$ and let $\epsilon_d=\frac{2(M-1)(\ell+2) }{d}$, and $V=V(Q)$.  Then 
\begin{enumerate}
	\item $\sum_{m=1}^d \sum_{v:v \in E_m(v)}  \frac{c_e}{|E_m|}=2c_e,$
	\item $\sum_{m=1}^d \sum_{v:v \in E_m(v)}  \frac{c_e(M-1)}{|E_m|}=2\epsilon_dc_e,$
	\item $\capt(e)=(1+\epsilon_d)c_e.$
\end{enumerate}
\end{lemma}
\begin{proof}
A vertex $v$ has $|E_m|$ edges in edge-layer $E_m(v)$ so in the cube there are $|V(Q)||E_m|$ vertex-edge pairs where the edge is in layer $m$ of the vertex. Therefore, by the symmetry of the cube, each edge has $|E_m|\frac{|V(Q)|}{|E(Q)|}=2|E_m|/d$ vertices for which it is in edge-layer $m$. So
	\[\sum_{m =1}^d\sum_{v:e \in E_m(v)}c_e|E_m|^{-1}=\sum_{m=1}^d\frac{2c_e|E_m|}{d}|E_m|^{-1}=2c_e.\] 
Similarly, 
	\begin{align*}
	&\sum_{m=1}^{\ell+2}\sum_{v:e \in E_m(v)}c_e(M-1)|E_m|^{-1}+\sum_{\stackrel{m=}{d-\ell-1}}^{d}\sum_{v:e \in E_m(v)}c_e(M-1)|E_m|^{-1}\\
	=&(M-1)\sum_{m=1}^{\ell+2}c_e \frac{2|E_m|}{d}|E_m|^{-1}+(M-1)\sum_{\stackrel{m=}{d-\ell-1}}^{d}c_e \frac{2|E_m|}{d}|E_m|^{-1}\\    
	=&~4\frac{(M-1)(\ell+2)}{d}c_e =2\epsilon_d c_e .
	\end{align*}
The third assertion follows by adding together the capacities from the two parts above.

\end{proof}
%

\section{Proof of Theorem 2}\label{sec:th2proof}
\subsection{Overview}

In Theorem \ref{th:2} we consider flows between all vertex pairs  in the network $\mathcal{N} \in \mathcal{G}(Q^d,C)$. We introduce the following terminology. For $u,v \in V(Q^d)$ we denote by $Q(u,v)$  the smallest cube containing both $u$ and $v$. If $d_Q(u,v)=k$ then $Q(u,v)$ has $2^k$ vertices. In set notation the vertices of $Q(u,v)$ are those subsets of $u \cup v$ which contain $u \cap v$. For $u,v \in V(Q^d)$ with $d(u,v)=k$ we denote by $\mathcal{N}^M(u,v)$ the subcube $Q(u,v)$ with edge-capacities defined in a manner analogous to equation (\ref{eq:capacity}). For edge $e$ in edge-layer $m$ in $Q(u,v)$, we let $c_e=2^{1-d}\binom{k}{m}^{-1}c_em^{-1}$ for $\ell+3\le m \le k-\ell-2$ and $c_e=M2^{1-d}\binom{k}{m}^{-1}c_em^{-1}$ for $m \le \ell+2$ or $m \ge k-\ell-1$. We let $S_u(v)=Q(u,v)\cap V_{\ell}(u). \text{ See figure }\ref{fig:x7}$. We consider vertex pairs with a small separation separately in Section \ref{sec:short}. So let $\mathcal{V}^{\text{near}}$ denote the set of unordered pairs of vertices $u,v$ such that $d(u,v)\le d/4$ and let $\mathcal{V}^{\text{far}}=\mathcal{V} \setminus \mathcal{V}^{\text{near}}$. 
\floatstyle{plain}
\restylefloat{figure}

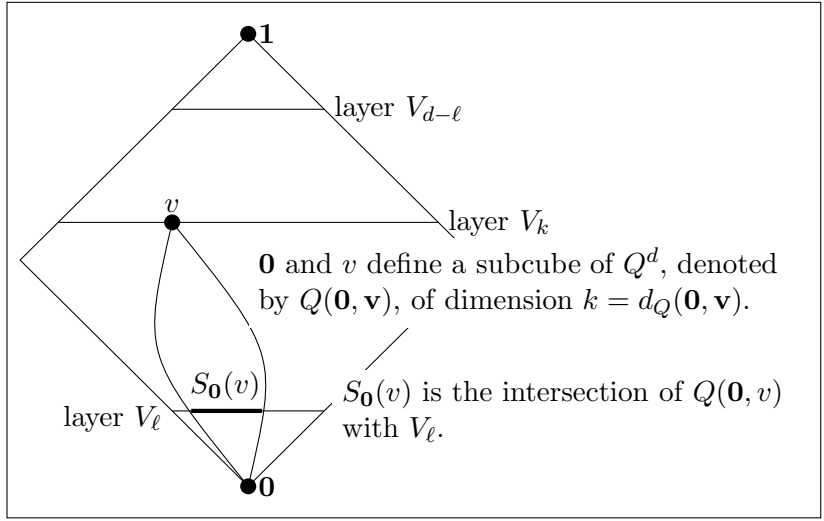
\begin{figure}[h]
	\centering
	\small
\begin{tikzpicture}[show background rectangle, >=stealth]
\fill[black] (0,0) circle(3pt) node[right]{$\bf{0}$};
\fill[black] (0,6) circle(3pt) node[right]{$\bf{1}$};
\fill[black] (-1,3.5) circle(3pt) node[above]{$v$};
\draw (0,0) -- (-3,3) -- (0,6) -- (3,3) -- (0,0);
\draw (-1,1) -- (1,1);
\draw (-1,0.9) node[left]{layer $V_{\ell}$ };
\draw (-1,5) -- (1,5)node[right,text width=4cm]{layer $V_{d-\ell}$ };
\draw (-2.5,3.5) -- (2.5,3.5)node[right]{layer $V_{k}$ };
\draw (0,0) .. controls (-1.4,1.75) and (-1.4,1.75).. (-1.0,3.5);
\draw (0,0) .. controls (0.4,1.75) and (0.4,1.75).. (-1.0,3.5);
\draw (0,2.7) node[right, fill=white, text width=7cm]
{$\bf{0}$ and $v$ define a subcube of $Q^d$, denoted by $Q(\bf{0},v)$, of dimension $k=d_Q(\bf{0},v)$.};
\draw (1.1,1.0) node[right, fill=white, text width=6cm]
{$S_{\mathbf{0}}(v)$ is the intersection of $Q(\mathbf{0},v)$ with $V_{\ell}$. };
\draw[ultra thick] (-0.75,1) -- (0.18,1);
\draw (-0.3,1.0) node[above]
{$S_{\mathbf{0}}(v)$ };
\end{tikzpicture}
\normalsize
\caption{\mdseries The $k-$cube defined by $\bf{0}$ and $v$ and the set $S_{\mathbf{0}}(v)$ \normalsize}	
	\label{fig:x7}
\end{figure}	
\floatstyle{boxed}
\restylefloat{figure}
The proof of Theorem \ref{th:2} uses much of the work from the first proof but we must be careful about how we show that we can escape from every vertex. For example if $d$ is even and we want to route a flow from $u$ to $v$ with $d_Q(u,v)=d/2$ we cannot just apply theorem \ref{th:1} to $Q(u,v)$ as the failure probability in Lemma \ref{lem:main1} is too great. This is not just a technical problem: if the open degree of $u$ is less than $d/2$ then there is some vertex $v$ with $d_Q(u,v)=d/2$  such that $u$ is an isolated vertex in $Q(u,v)$. However, by Lemma \ref{lem:main1} we know that for some constant $M$, with high probability, we can route a balanced flow to $V_{\ell}$ in $\mathcal{N}^M(u)$. Lemma \ref{lem:decomp} is a deterministic result that uses the symmetry of the cube to divide up this flow and allocate the parts to the different flows from $u$ to  all the sets $S_u(v)$ in a balanced manner to achieve the flows we need.

In Lemma \ref{lem:SuvSvu} we show that with high probability we can route $(1+\epsilon)^{-1}d^{-2}$-near-balanced flows of volume $\Ex[C]2^{1-d}$ between $S_u(v)$ and $S_v(u)$ for all $\{u,v\} \in \mathcal{V}^{\text{far}}$. We further show that the capacity required in an edge $e$ to achieve all these flows simultaneously is less than $c_e$. We then `stitch' these flows together using Lemma \ref{lem:stitching} in the proof of Lemma \ref{lem2:far}, which covers flows between all vertex pairs in $\mathcal{V}^{\text{far}}$.

In Lemma \ref{lem:short} we show that for all pairs $\{u,v\} \in \mathcal{V}^{\text{near}}$ we can route simultaneous flows via distant vertices and that the total capacity required for edge $e$ in this network is less than $\epsilon c_e$.

The theorem follows from these lemmas.

\subsection{Flows close to a source}
In Lemma \ref{lem:main1} we showed that there exist constants $\rho>0$ and $M>0$ such that the following holds. With failure probability $2^{-\rho d}$:- for all vertices $u$ there exists a balanced flow $f_u$ of volume $\Ex[C]$ in the network $B_{u}^M({\ell}+2)$ from $u$ to $V_{\ell}(u)$. We further showed in Lemma \ref{lem:qsumcap} that for each edge $e$ of $Q$ the sum of the absolute values of the flows along $e$ from the $f_u$ is at most $\epsilon_d c_e$ where $\epsilon_d=\frac{2(\ell+2)M}{d}=O(d^{\kappa-1})$

We now want to prove the deterministic result that if $\mathcal{N}^*$ is a network (on $Q^d$) with the above property then for every vertex $u$ we can decompose the flow from $u$ to $V_{\ell}(u)$ into balanced flows of volume $2^{-d}$ from $u$ to $S_u(v)$ for all $v \in V_k$ and all $d/4 < k \le d$, and that this can be accomplished using only a small ($\epsilon_d$) proportion of the capacity of any edge. (Note that for presentation purposes we decompose flows of volume 1 from $u$ to $V_{\ell}(u)$ implying a different choice of the constant $M$).

\begin{lemma}
\label{lem:decomp}
Let $f$ be a balanced flow of volume 1 from $\mathbf{0}$ to $V_{\ell}$ in $Q^d$. Then $f$ may be decomposed as $g + \sum_v f^v$ where the sum is over $v\in Q^d$ with $d(\mathbf{0},v)>d/4$ and where $f_v$ is a balanced $\mathbf{0}$-$S_{\mathbf{0}}(v)$ flow of volume $2^{-d}$ and there is no cancellation in the sum.

\end{lemma}
\begin{proof}
Let $\mathbf{0} \in V(Q^d)$, and let $f^{\mathbf{0}}$ be a balanced flow of volume 1 in $B^{*M}_{\ell+2}(\mathbf{0})$ from $\mathbf{0}$ to $V_{\ell}(\mathbf{0})=V_{\ell}$. We first show that there are flows $f^{\mathbf{0}v}$ for each vertex $v$ such that $\{\mathbf{0},v\} \in \mathcal{V}^{\text{far}}$ such that   i) $f^{\mathbf{0}v}$ is a balanced flow of volume $2^{-d}$ from $\mathbf{0}$ to $S_{\mathbf{0}}(v)$, and that ii) $\sum_{v:\{u,v\} \in \mathcal{V}^{\text{far}}}|f^{\mathbf{0}v}| \le f^{\mathbf{0}}(e)$ for each edge $e \in E(Q^d)$. 
	
The flow $f^{\mathbf{0}}$ can be decomposed into flows along paths and around cycles (see for example \cite{AhujaMagnantiOrlin}). For each $w \in V_{\ell}(\mathbf{0})$ we denote by $F_w$ the sum of all the flows along $\mathbf{0}w$ paths that end at vertex $w$. $F_w$ is therefore a flow of volume $|V_{\ell}|^{-1}$ from $\mathbf{0}$ to $w$. For $d/4<k\le d$ and $v \in V_k(\mathbf{0})$ and for each $w \in S_{\mathbf{0}}(v)$ we now allocate a flow of volume $\binom{k}{l}^{-1}2^{-d}$ out of the total flow $F_w$ as commodity $\mathcal{K}_{\mathbf{0},v}$.  $|S_{\mathbf{0}}(v)|=\binom{k}{l}$ so this flow allocation is a balanced flow of volume $2^{-d}$ from $\mathbf{0}$ to $S_{\mathbf{0}}(v)$. We now perform this allocation for all $v \in V_k$ for all $d/4<k\le d$ and we need to check that the total volume of flow allocated is at most the original flow to each vertex in $V_{\ell}(\mathbf{0})$, namely $|V_{\ell}|^{-1}$.

For a particular $k$, each $w \in V_{\ell}(u)$ is in $\binom{d-\ell}{k-\ell}$  sets $S_{\mathbf{0}}(v)$ so the total flow required at $w$ is
\begin{align*}
\sum_{k=\left\lfloor d/4\right\rfloor+1}^d\binom{d\minus \ell}{k \minus \ell}\binom{k}{\ell}^{-1}2^{-d}
&=\sum_{k=\left\lfloor d/4\right\rfloor+1}^d\binom{d}{k}\binom{k}{\ell}\binom{d}{\ell}^{-1}\binom{k}{\ell}^{-1}2^{-d}\\
&=|V_{\ell}|^{-1}\sum_{k=\left\lfloor d/4\right\rfloor+1}^d\binom{d}{k}2^{-d}\\
&\le |V_{\ell}|^{-1}.
\end{align*}
Hence we have shown that there exists a balanced flow of volume $2^{-d}$ of commodity $\mathcal{K}_{\mathbf{0},v}$ from $\mathbf{0}$ to $S_{\mathbf{0}}(v)$ in $B^{*M}_{\ell+2}({\mathbf{0}})$ for all $v$ in layers $\left\lfloor d/4\right\rfloor+1$ to $d$. 

The flow volume in edge $e$ from all flows from $u$ to $S_u(v)$ for all $\{u,v\} \in \mathcal{V}^{\text{far}}$ is at most the sum of the capacities of edge $e$ in the balls $B^M_{m}(u)$ for all $u$ and all $1 \le m \le \ell+2$ which is at most $\epsilon_d c_e$ by Lemma \ref{lem:qsumcap}.
\end{proof}

Putting Lemma \ref{lem:main1} together with Lemma \ref{lem:decomp} we get

\begin{lemma}
\label{lem:main21}
Fix $\epsilon>0$ and suppose that $C$ is a random variable with $\Pro[C>0]>\frac12$. Then there exists constants $\rho>0$ and $M>0$ such that the following holds with failure probability $2^{-\rho d}$. For $\mathcal{N} \in \mathcal{G}(Q^d,C)$, for all vertex pairs $\{u,v\} \in \mathcal{V}^{\text{far}}$ there exists a balanced flow of volume $2^{-d}$ from $u$ to $S_u(v)$ in $B^M_{\ell+2}(u)$, such that the total flow volume (with no cancellation) from all such flows in any edge $e$ in $\mathcal{N}$ is at most $\epsilon \; c_e$.

\end{lemma}
\subsection{The flow in the middle part of the flow}

\begin{lemma}
Fix $\epsilon>0$ and suppose that $C$ is a random variable with $\Pro[C>0]>\frac12$ and let $\mathcal{N} \in \mathcal{G}(Q^d,C)$. With failure probability $O(e^{-\Omega(d^2)})$, for all $\{u,v\} \in \mathcal{V}^{\text{far}}$ there exist simultaneous $d^{-2}$-near-balanced flows of volume $(1+\epsilon)^{-1}\Ex[C]2^{1-d}$ between $S_u(v)$ and $S_v(u)$ in the network $\mathcal{N}$.
\label{lem:SuvSvu}
\end{lemma}
\begin{proof}
Lemma \ref{lem:main2} showed that with failure probability $e^{-\Omega(d^2)}$ there exists a $d^{-2}$-near-balanced flow of volume $\Ex[C]$ between $V_{\ell}(\mathbf{0})$ and $V_{\ell}(\mathbf{1})$ in $\mathcal{N}(u)$. Applying this result to  the networks $\mathcal{N}^M(u,v)$ for all vertex pairs $u,v$ with  $d(u,v)>d/4$ and taking the union bound  we see that with failure probability $2^{2d}e^{-\Omega(d^2)}=e^{-\Omega(d^2)}$ there exists simultaneous $d^{-2}$-near-balanced flows of volume $(1+\epsilon)^{-1}\Ex[C]$ from $S_u(v)$ to $S_v(u)$ in $\mathcal{N}^M(u,v)$ for every vertex pair $u,v \in \mathcal{V}^{\text{far}}$. We now scale these flows by a factor $2^{1-d}$ and superpose these simultaneously in the network $\mathcal{N}$. The capacity required is at most the capacity calculated from the sum of the capacities of all networks $\mathcal{N}^M(u,v)$ for all $\{u,v\} \in \mathcal{V}^{\text{far}}$. Consider first all sub-cubes $Q(u,v)$ for which $d(u,v)=k$ (counting sub-cube $Q(u,v)$ and $Q(v,u)$ once). There are $2^{d-1}\binom{d}{k}$ such sub-cubes and they each have $\binom{k}{m}$ edges in edge-layer $m$ from $u$ so each edge is in $\frac{2^{d-1}\binom{d}{k}\binom{k}{m}}{d2^{d-1}}$ such sub-cubes in edge layer $m$.  Thus an edge $e$ has to contribute a total capacity at most $\binom{d}{k}\binom{k}{m} d^{-1} \frac{2^{1-d}c_e}{\binom{k}{m}}$ to the flows of volume $2^{1-d}(1+\epsilon)^{-1}\Ex[C]$ in the networks $\mathcal{N}^M(u,v)$ with $d_Q(u,v)=k$ for which it is in layer $m$. Thus the total capacity required in edge $e$ is at most
\begin{align*}
\sum_{k>d/4}^{d}\sum_{m=\ell}^{k-\ell-1}\binom{d}{k}\binom{k}{m} d^{-1} \frac{c_e2^{1-d}}{\binom{k}{m}} &\le \sum_{k=1}^{d}\sum_{m=1}^{k}\binom{d}{k}2^{1-d}d^{-1} c_e\\
&=2^{1-d}\sum_{k=1}^{d}\frac{k}{d}\binom{d}{k} c_e\\
&=2^{1-d}\sum_{k=0}^{d-1}\binom{d\minus 1}{k} c_e=c_e.
\end{align*}
\end{proof}

\subsection{Flows between vertex pairs with large separation}
Let $\mathcal{N} \in \mathcal{G}(Q^d,C)$ and denote by $\Phi_{\text{far}}$ the maximum uniform flow volume when $\mathcal{V}$ is taken as $\mathcal{V}^{\text{far}}$
\begin{lemma}\label{lem2:far}
Fix $\epsilon>0$ and suppose that $\Pro[C>0]>\frac12$. Then as $d\rightarrow \infty$,
\[\Pro[2^{d-1}\Phi_{\text{far}}\ge(1-\epsilon)\Ex[C]]\rightarrow 1.\]
\end{lemma}
\begin{proof}
Lemma \ref{lem:main21} showed that given $\epsilon>0$ there exists constants $\rho>0$ and $M>0$ such that with failure probability $2^{-\rho d}$ the following holds. There exist simultaneous balanced flows of volume $2^{-d}$ in $B_u^M(\ell+2)$ between $u$ and $S_u(v)$ for all pairs $\{u,v\}$ in $\mathcal{V}^{\text{far}}$ such that the total flow (with no cancellation) from all such flows occurring in any edge $e$ in $\cN$ is less than $\epsilon c_e$. Lemma \ref{lem:SuvSvu} showed that with failure probability $e^{-\Omega(d^2)}$ for all vertex pairs $\{u,v\} \in \mathcal{V}^{\text{far}}$ there exist simultaneous $d^{-2}$-near-balanced flows of volume $(1+\epsilon)^{-1}\Ex[C]2^{1-d}$ between $S_u(v)$ and $S_v(u)$ in $\cN$. 

Thus we achieve the flows we require by scaling the flows and capacities from Lemma \ref{lem:main21} by $\Ex[C]$ and `stitching' them together with the flows from Lemma \ref{lem:SuvSvu} using the results from Lemma \ref{lem:stitching}. The capacity of edge $e$ in the superimposed networks is shown to be at most $(1+\epsilon)c_e$ by Lemma \ref{lem:qsumcap}.
\end{proof}

\subsection{Vertex pairs with small separation}\label{sec:short}
To route flows between pairs of vertices that are `near' ($d(u,v)\le d/4$) we route flows to distant vertices and back again. We have already shown that, w.h.p. these flows exist, and the number of `near' vertex pairs is small so the flows can be accommodated in a small part of the capacities of any edge.

\begin{lemma}\label{lem:short}
Fix $\epsilon>0$ and suppose that $\Pro[C>0]>\frac12$. Let $\mathcal{N} \in \mathcal{G}(Q^d,C)$. Then there exists a constant $\rho>0$ such that with failure probability $2^{-\rho d}$ there are simultaneous flows of volume $2^{1-d}\Ex[C]$ between vertices $u$ and $v$ for all $\{u,v\} \in \mathcal{V}^{\text{near}}$  in the network $\mathcal{N}$ with the edge-capacities of $\mathcal{N}$ scaled by $\epsilon$.
\end{lemma}
\begin{proof}
We want a flow of volume $2^{1-d}\Ex[C]$ between two vertices $u,v$ a distance apart less than $d/4$. The idea is to route half the flow from $u$ to $\overline{u}$ (the antipodal point of $u$)  and then back to $v$ and half from $u$ to $\overline{v}$ and then back to $v$. From Lemma \ref{lem2:far}, with high probability, all of these flows  ($u$ to $\overline{u}$, $\overline{u}$ to $v$, $v$ to $\overline{v}$, $\overline{v}$ to $u$) with volumes $(1-\epsilon)^{-1}\Ex[C]$  exist in networks $\mathcal{N}(u,\overline{u})$, $\mathcal{N}(u,\overline{v})$ for all $\{u,v\}$. The number of such vertex pairs is at most $2^d (d/4){|V_{\left\lfloor d/4\right\rfloor}|}\le 2^d d(4e)^{d/4} \le 2^{1.9 d} $ for large $d$. Hence, by the symmetry of the cube, the total flow in any edge $e$ from all these flows is at most
\[ 2d 2^{1-d} (1-\epsilon)^{-1}\Ex[C] \frac{2^{1.9 d}}{d 2^{d-1}}\rightarrow 0.\] 
\end{proof}


\section{Concluding Remarks}\label{sec:conclusion}
We have investigated uniform multicommodity flows in the cube with random edge-capacities distributed like some given random variable $C$. There are two natural directions for further investigation. Firstly, this paper has been restricted to the case $\Pro[C>0]>1/2$ which ensures that, with high probability, the network $\mathcal{N}$ is connected, and in this case we have been able to tell a full story. The component structure of the network in the case $C \sim \Ber(p)$ for $0 <p\le 1/2$ is more complicated.   (See \cite{component} for a full analysis). Multicommodity flows in the largest component of such networks are investigated in another paper \cite{secondpaper}. We remark that for $d \ge 4$ there exists a disconnected subgraph of $Q^d$ in which each antipodal pair of vertices is connected by a path. Such a subgraph could support a non-zero antipodal multicommodity flow despite being disconnected. However for $p<1/2$, the probability of isolated vertices tends to 1 and hence the probability of a non-zero uniform multicommodity flow tends to zero.

Secondly, this paper looked at undirected networks but it is also interesting to look at directed networks formed by replacing each undirected edge of $Q^d$ with a pair of opposingly directed edges with (possibly identical, possibly independent) random edge-capacities. We study this in \cite{PWThesis}.


\addcontentsline{toc}{section}{Bibliography}
%

\end{document}